\documentclass{article}

\usepackage{graphicx, amsmath, amsthm, amssymb} 
\usepackage{todonotes}

\newtheorem{thm}{Theorem}[section]
\newtheorem*{repskewthm}{Theorem 1.4}

\newtheorem{conj}[thm]{Conjecture}

\newtheorem{lem}[thm]{Lemma}
\newtheorem{corollary}[thm]{Corollary}
\newtheorem{claim}[thm]{Claim}

\newtheorem{fact}[thm]{Fact}

\newcommand{\remove}[1]{}

\newcommand\eps{\varepsilon}

\renewcommand\ge{\geqslant}
\renewcommand\le{\leqslant}

\renewcommand{\epsilon}{\eps}

\newcommand{\rom}[1]{\uppercase\expandafter{\romannumeral #1\relax}}

\title{Completing the proof of the Liebeck--Nikolov--Shalev conjecture}
\author{Noam Lifshitz}
\date{}

\begin{document}

\maketitle
\begin{abstract}
    Liebeck, Nikolov, and Shalev conjectured the existence of an absolute constant $C>0$, such that for every subset $A$ of a finite simple group $G$ with $|A|\ge 2$,  there exists $C\log|G|/\log|A|$ conjugates of $A$ whose product is $G$. This paper is a companion to \cite{GLPS}, and together they prove the conjecture. 

   To prove the conjecture, we establish the following skew-product theorem. We show that there exists \( c > 0 \) such that for all \( \epsilon > 0 \) and subsets \( A, B \subseteq G \) of finite simple groups of Lie type, if \( |B| < |G|^{1 - \epsilon} \), then \( |A^{\sigma} B| > |B||A|^{c \epsilon} \) for some \( \sigma \in G \). This result, along with its more involved analogue for alternating groups, constitutes the main contribution of this paper. 
   
Our proof leverages deep results from character theory alongside the probabilistic method.
\end{abstract}

\section{Introduction}

In the study of finite groups and their generating sets $A\subseteq G$, the field of `growth in finite groups' explores whether every element of $G$ can efficiently be expressed as a product of few elements from $A$. A striking result in this area, due to Liebeck and Shalev~\cite{liebeck2001diameters} demonstrates that elements of finite simple groups possess an optimal decomposition using elements from any conjugacy class. 

In light of this, Liebeck, Nikolov, and Shalev \cite{liebeck2012product} proposed a broader conjecture that encompasses the Liebeck-Shalev theorem and other significant results in group theory, such as those found in Liebeck and Pyber~\cite{liebeck2001finite}. This paper is a companion to Gill, Lifshitz, Pyber, and Szab\'{o}~\cite{GLPS}, and together they complete the proof of the conjecture.

\subsection{Growth in finite groups}
The product of two sets is defined via $AB:= \{ab: a\in A,b\in B\}$. Additioninally, the $i$th power of a set is given by $A^{i} = \{a_1\cdots a_i:\,a_1,\ldots, a_i\in A\}.$ The minimal $i$ for which $A^i=G$ is referred to as the \emph{covering number} of $A$. Liebeck and Shalev~\cite{liebeck2001diameters} established the existence of an absolute constant $C>0$, such that for every subset $A\subseteq G$ of a finite simple group, with  $|A|\ge 2$, the covering number of $A$ is at most $\frac{C\log |G|}{\log |A|}.$ This result matches the trivial lower bound on the covering number up to a constant factor.

The problem of determining the covering numbers of generating sets $A$ has been extensively studied also in the non-normal setting. Pyber and Szab\'{o}~\cite{pyber2016growth}, and independently Breuillard, Green, and Tao \cite{breuillard2011approximate}, showed that in a  finite simple group of Lie type of rank 
$r$, the covering number of a generating set $A$ is at most $\left( \frac{\log |G|}{\log |A|}\right)^{C},$ where $C =C(r)$ depends only on the rank. This generalized an earlier result for the special case of $SL_2(\mathbb{F}_p)$ due to Helfgott \cite{helfgott2008growth}. The well-known Babai's diameter problem aims to obtain a similar bound for all finite simple groups, with $C$ being an absolute constant.  

\subsection{Decomposing a group as a product of copies of a set}

The related problem of decomposing $G$ as a product of `copies' $A^{\sigma}:= \{ \sigma^{-1} a \sigma: a\in A\}$ of $A$ is also well studied. For instance, Liebeck and Pyber~\cite{liebeck2001finite} proved that if $G$ is a finite simple group of Lie type of characteristic $p$, then it is a product of $25$ of its Sylow $p$-subgroups, a result later improved from 25 to 5 by Babai, Nikolov, and Pyber~\cite{bnp}, and then to 4 independently by Smolenski~\cite{smolensky2016products} and Garonzi, Levy, Mar\'{o}ti, and Simion~\cite{garonzi2016minimal}. 

This led Liebeck, Nikolov, and Shalev~\cite{liebeck2010conjecture} to conjecture the existence of a constant $C>0$, such that for every subgroup $H$ of a finite simple group $G$ there exist  $O\left(\frac{\log |G|}{\log |H|}\right)$ conjugates of $H$ whose product is the whole group. They then went on in \cite{liebeck2012product} to state an even stronger conjecture that contains both the Liebeck--Shalev theorem~\cite{liebeck2001diameters} and their earlier conjecture as special cases. 

\begin{conj}[\cite{liebeck2012product}]
There exists an absolute constant $C>0,$ such that if $G$ is a finite simple group and $A\subseteq G$ has size $\ge 2$, then there exist $m\le C\frac{\log |G|}{\log |A|},$ and $\sigma_1,\ldots, \sigma_m\in G$ with $A^{\sigma_1}\cdots A^{\sigma_m} =G.$ 
\end{conj}

Their conjecture was proved for finite simple groups of Lie type in the bounded rank regime by Gill, Pyber, Short, and Szabo \cite{gill2013product}. More recently, Gill, Lifshitz, Pyber, and Szabó~\cite{GLPS} established the conjecture for arbitrary finite simple groups in the special case, where 
$A$ has size at least $|G|^c$ for an absolute constant $c>0$ (See Theorem \ref{thm:GLPS_intro} below).

This paper, together with the results from \cite{GLPS}, completes the proof of the Liebeck--Nikolov--Shalev conjecture and establishes the following stronger statement.

\begin{thm}\label{thm:main intro}
There exists an absolute constant $c>0,$ such that if $G$ is a finite simple group and $A\subseteq G$ has size $\ge 2$, then for every integer $i$ there exist $\sigma_1,\ldots, \sigma_i\in G$ with $|A^{\sigma_1}\cdots A^{\sigma_i}| \ge \min(|A|^{ci}, |G|).$ 
\end{thm}

\subsection{Skew Product theorems}

Upper bounds on covering numbers are often proved with the help of product theorems. Indeed, both Pyber and Szab\'{o}~\cite{pyber2016growth} and Breuillard, Green, and Tao~\cite{breuillard2011approximate} deduced their upper bound on the covering number by proving a product theorem first.  Their product theorem states that for every rank $r$, there exists $\epsilon(r)$, such that for every finite simple group of rank $r$ and every generating subset $A\subseteq G$, $|A^3| \ge \min(|A|^{1+\epsilon}, |G|)$. 
Gill, Pyber, Short, and Szabo~\cite{gill2013product} later showed that for every finite simple groups of Lie type of rank $r$ for all sets $A\subseteq G$ either $|A^{\sigma} A|>|A|^{1+\epsilon}$ for some $\epsilon = \epsilon(r)$ and $\sigma \in G$ or $A^3=G.$ They also showed a similar result holds for conjugacy classes in arbitrary finite simple groups. Recently, Skresesanov~\cite{Skresanov2024expanders} showed that if $A$ is a conjugacy class, then either $|A^2|>|A|^{1+\epsilon}$ for some $\epsilon = \epsilon(r)$ or $A^2 = G\setminus\{1\}.$

This paper is a companion to Gill, Lifshitz, Pyber, Szab\'{o}~\cite{GLPS} who proved the following growth result, which they termed a skew product theorem. They showed that there exists an absolute constants $c>0$, such that for every finite simple group $G$ and every  $A\subseteq G$ either there exists $\sigma \in G$ with $|A^{\sigma} A|\ge |A|^{1+\epsilon}$ or there exist $C$ conjugates of $A$ whose product is the whole group. They then deduced the following as a corollary.

\begin{thm}[\cite{GLPS}\label{GLPS theorem}]\label{thm:GLPS_intro}\label{thm: GLPS main theorem}
For every $c>0$ there exists $C>0$, such that if $G$ is a finite simple group and $A\subseteq G$ has size $\ge \max(|G|^{c},2),$ then there exist $C$ conjugates of $A$ whose product is the whole group.
\end{thm}

Thus, their result establishes optimal growth for sets of size $\ge |G|^{c}$. Prior to this work Dona~\cite{dona2024writing} showed that for every $\epsilon>0$, there exists $N_{\epsilon}>0$, such that if $G$ is a finite simple group of Lie type and $A\subseteq G$ is with $|A|\ge 2$, then there are $N_{\epsilon} \left(\frac{\log{|G|}}{\log |A|}\right)^{1+\epsilon}$ conjugates of either $A$ or $A^{-1}$ whose product is the whole group.

In this paper, we complement \cite{GLPS} by proving the following skew-product theorem, which is suitable for the growth of smaller sets. The following theorem applies only to finite simple groups of Lie type.

\begin{thm}\label{thm:skew product theorem_intro}\label{thm:skew product theorem}
     There exist $c>0$, such that the following holds. Let $G$ be a finite simple group of Lie type, let $\epsilon >0$, and suppose that sets $A,B\subseteq G$ satisfy $|B|<\min(|G|^{1 -\epsilon}, c|G|).$ Then 
    \[
    |A^\sigma B| \ge |A|^{c\epsilon}|B|
    \]
    for some $\sigma \in G.$
\end{thm}

For alternating groups, we were able to obtain a similar result, with the exception of one case where $A$ has a special structure.  We then manage to address the structured case separately (See Lemma \ref{lem: growth or small support} below). This allows us to obtain the following theorem, which completes the proof of the Liebeck--Nikolov--Shalev conjecture for alternating groups in conjunction with Theorem~\ref{thm: GLPS main theorem}.

\begin{thm}\label{thm:alternating growth}
    There  exists an absolute constant $c>0$, such that if $G$ is an alternatating group, then for every subsets $A_1,\ldots, A_i\subseteq G$, there exist $\sigma_1,\ldots ,\sigma_i \in G$ with $|A_1^{\sigma_1} A_2^{\sigma_2}\cdots A_i^{\sigma_i}|\ge \min(|A_1|\cdots |A_i|, |G|)^{c}.$
\end{thm}

\subsection{Method}

A subset \( A \subseteq G \) is \emph{normal} if \( A^{\sigma} = A \) for all \( \sigma \in G \). Character theory is extensively used in the study of the growth of normal sets (see, e.g., \cite{liebeck2001diameters, larsen2008characters, shalev2009word, larsen2024uniform}).

One of the main tools that we develop in this paper is the following result that opens the door for deducing skew-product theorems for non-normal sets via input from character theory. 

\begin{lem}\label{lem:from character theory to growth}
Let $G$ be a finite group, $\epsilon >0$, and \( K < |G|^{\epsilon/2}.\) Suppose that $S$ is a set of irreducible characters for $G$ with $\sum_{\chi \notin S} \chi(1)^2 < |G|^\epsilon.$ Suppose additionally that for every $\chi \in S$ and random independent elements $x,y \sim A$, we have
\[|\mathbb{E}[\chi (y^{-1}x)]| \le \frac{\chi(1)}{K}.\] 
Then for all sets $B$ with $|B|<|G|^{1-2\epsilon},$ there exists $\sigma\in G$ with 
\[
|A^{\sigma} B| > \frac{K|B|}{1+ |G|^{-\epsilon/2}}. 
\]
\end{lem}

To obtain Theorem \ref{thm:skew product theorem_intro}, our main idea is to combine Lemma~\ref{lem:from character theory to growth} with input from character theory. In particular, we make use use of the recent character bounds obtained by Larsen and Tiep~\cite{larsen2024uniform}, which build upon the works of Guralnick, Larsen, and Tiep~\cite{guralnick2020character,guralnick2024character}. Their deep sequence of papers utilizes Deligne--Lustig theory, to obtain strong character bounds for the finite classical groups.   

To complete the proof of Theorem \ref{thm:main intro} for alternating groups, we rely on more classical character bounds due to M\"{u}ller and Schlage-Puchta~\cite{muller2007character}. Specifically, for given sets \( A \) and \( B \), we prove variuos lemmas showing that either there is growth of the form \( |A^{\sigma} B| > |A|^c |B| \) for an absolute constant \( c > 0 \), or the sets exhibit a specific structure. We then prove Theorem \ref{thm:alternating growth} for structured sets using the probabilistic method. 

\section{Preliminaries}
\subsection{Character bounds}

Let $\sigma \in S_n$ be a permutation. The \emph{support} of $\sigma$ is the set of coordinate $i$ with $\sigma(i)\ne i$. Thus, if $\sigma\in S_n$ has $f$ fixed points, then the support size of  $\sigma$ is $n-f$. The following character bound for the symmetric group is due to M\"{u}ller and Schlage-Puchta~\cite{muller2007character}.  
\begin{thm}\label{thm:Schlage-Puchta}
    There exists an absolute constant $c>0$, such that the following holds. Suppose that $\sigma \in S_n$ has support size $s$. Then $\chi(\sigma) \le \chi(1)^{1 - \frac{c s}{n \log n}}$ for every character $\chi$ of $S_n$.
\end{thm}

We will actually need the corresponding result for $A_n$. 
 \begin{thm}\label{thm:Schlage-Puchta A_n}
    There exists absolute constants $c,n_0>0$, such that the following holds. Let $n>n_0, s\le n$ and suppose that $\sigma\in A_n$ has support size $s$. Then $\chi(\sigma) \le \chi(1)^{1 - \frac{c s}{n \log n}}$  for every irreducible character $\chi$ of $A_n$.
\end{thm} 
\begin{proof}
    Let $\lambda\vdash n$ be a partition and $\lambda'$ be its conjugates obtained by reversing the roles of its columns and rows.  Recall that whenever $\lambda \ne \lambda'$ the character $\chi_{\lambda}$ restricts to an irreducible representation of $A_n$. For those characters the statement follows immediately from Theorem \ref{thm:Schlage-Puchta}. For the other characters with $\lambda = \lambda'$, the character $\chi_{\lambda}$ restricts to the sum of two irreducible representation $\chi_{\lambda}'$ and $\chi_{\lambda}'',$ where $\chi_{\lambda}''(\sigma) =  \chi_{\lambda}'(\tau \sigma \tau^{-1})$ for all $\tau \in S_n \setminus A_n.$ Now for each conjugacy class either $\sigma^{S_n} = \sigma^{A_n},$ and then the statement again follows immediately from Theorem \ref{thm:Schlage-Puchta} or $\sigma^{S_n}\ne \sigma^{A_n}$.
    Now for $\sigma$ with $\sigma^{S_n}\ne \sigma^{A_n}$ Larsen and Tiep~\cite[Theorem 2]{larsen2023squares} showed that $\chi_{\lambda'}(\sigma)\le \sqrt{\chi(1)},$ provided that $n$ is sufficiently large. This completes the proof, provided that $c\le 1/2.$
\end{proof}
The following corresponding character bound for finite simple groups of Lie type is due to 
Gluck~\cite{gluck1995sharper}.

\begin{thm}[\cite{gluck1995sharper}]\label{thm:Gluck}\label{lem:When s is small}
    There exists $c>0$, such that the following holds. Let $G$ be a finite simple group of Lie type over $\mathbb{F}_q$, $\chi$ an irreducible character and $g\in G\setminus\{1\}$. Then \[|\chi(g)|\le \chi(1)/q^{c}.\] 
\end{thm}

Larsen and Tiep~\cite{larsen2024uniform} obtained the following result for finite classical groups, which improves upon Theorem \ref{thm:Gluck} in the high rank regime.  
\begin{thm}[\cite{larsen2024uniform}]\label{thm:larsen--tiep}
    There exists $c>0$, such that the following holds. Let $G$ be a finite simple group of Lie type, $\chi$ be an irreducible character and $g\in G$. Then \[|\chi(g)|\le \chi(1)^{1- c\frac{\log |g^{|G|}|}{\log|G|}}.\] 
\end{thm}
The works of Larsen and Tiep build upon the works of Guralnick, Larsen, and Tiep~\cite{guralnick2020character, guralnick2024character}, which make heavy use of Deligne--Lustig theory.

\subsection{Facts about representations and conjugacy classes}

We say that $G$ is a \emph{classical finite simple group} $Cl_n(\mathbb{F}_q)$ to mean that $G$ is the finite simple groups of special linear isometries of a nondegenerate sesquilinear form, defined over a vector space of dimension $n$ over the field $\mathbb{F}_q$, and modded out by its center (see the book of Kelidman and Liebeck \cite{kleidman1990subgroup} for the definitions, or alternatively have $SL_n(\mathbb{F}_q)$ in mind, while taking a leap of faith properties in their properties that we recollect in this section). 
We require the following well known result concerning the number of conjugacy classes for classical finite simple groups (see e.g. \cite[Theorem 7.5.1]{de2024conjugacy}) 
\begin{fact}\label{fact:number of conjugacy classes}
    There exists $C>0$, such that the following holds. Let $G$ be a classical finite simple group $Cl_{n}(\mathbb{F}_q).$ Then $G$ has $\le q^{Cn}$ congugacy classes. 
\end{fact}

We shall also require the following classical result due to Landazuri and Seitz~\cite{landazuri1974minimal}.  
\begin{fact}\label{fact:Quasirnadomness of simple classical groups}
    There exists $c_1>0$, such that if $G$ is a classical finite simple group $Cl_n(\mathbb{F}_q)$, then every nontrivial representation of $G$ has dimension $\ge q^{c_1 n}.$
\end{fact}

In particular, the special case of representations induced from the trivial representation of Fact~\ref{fact:Quasirnadomness of simple classical groups} implies that there exists an absolute constant $c>0$, such that every subgroup of a classical finite simple group has index $\ge q^{cn}.$ The centralizer theorem therefore implies the following well known fact. 

\begin{fact}\label{fact:sizes of conjugacy classes}
There exists $c>0$, such that every conjugacy class of a classical finite simple group $Cl_{n}(\mathbb{F}_q)$ has size $\ge q^{cn}.$ 
\end{fact}

We make use of the following result due to Larsen, Tiep, and Taylor that upper bounds the number of representations of a given order of magnitude.  
\begin{thm}[{\cite[Theorem 9.1]{larsen2023character}}]\label{thm: Counting low dimensional representations}
    There exists an absolute constant $C>0,$ such that for all $D>0$ and every classical finite simple group $G=Cl_n(\mathbb{F}_q)$ the number of representations of $G$ of dimension at most $D$ is at most $D^{C/n}.$ 
\end{thm}

We also need the following theorem of Liebeck and Shalev \cite{liebeck2004fuchsian}
\begin{lem}[{\cite[Theorem 1.1 and  Corollary 2.7]{liebeck2004fuchsian}}]\label{lem:Liebeck-shalev zeta}
    For every $\epsilon>0$ there exists $n_0>0$, such that if $n>n_0$ and $D>1$, then $A_n$ has at most $(1+\epsilon)D^{1+\epsilon}$ representations whose dimension is at most $D.$ 
\end{lem}

\subsection{The generalized Frobenius formula of \cite{GLPS}}

Before stating the generalized Frobenius formula of Gill, Lifshitz, Pyber, and Szab\'{o}, we introduce some notations and recall some basic facts from the representation theory of finite groups. These facts can be found e.g. in \cite{faraut2008analysis}.
\subsubsection*{Functions on groups}
We shall write $x\sim A$ to denote that $x$ is chosen uniformly out of $A$. For a function on a finite group \(f\in \mathbb{C}[G]\) we therfore write  $\mathbb{E}_{x\sim A}f$ for $\frac{1}{|A|}\sum_{x\in A}f(x)$. We also write $\mathbb{E}[f]$ as a shorthand for $\mathbb{E}_{x\sim G}[f(x)]$. We write $f^{\sigma}$ for the outcome of the right action of $\sigma \in G$ on $f\in \mathbb{C}[G]$, i.e. 
  $f^{\sigma}(\tau ) = f(\tau\sigma)$. The space of functions on $G$ is equipped with the normalized inner product  \[\langle f, g\rangle = \frac{1}{|G|}\sum_{\sigma}f(\sigma)\overline{g(\sigma)}.\]
  We write $\|f\|_p$ for the $L_p$-norm of $f$ given by $\|f\|_p^p = \mathbb{E}[|f|^p]$ and write $\|f\|_{\infty}$ for the maximal value of $|f|$. 

  \subsubsection*{Probaility measures and density functions}
  Every probability measure $\mu$ on a finite group corresponds to a function $f\colon G\to [0,\infty)$ with $\|f\|_1 = 1$ given by $f(\sigma) = |G|\mu(\sigma).$ The function $f$ is the Radon--Nikodym derivative of $\mu$ with respect to the uniform measure. We call $f\colon G\to [0,\infty)$  with $\|f\|_1 = 1$ a \emph{density function}. Density functions are in one to one correspondence with probability measures. If $f$ is the density function corresponding to $\mu$, then we have 
  \[
  \langle g,f \rangle = \mathbb{E}_{x\sim \mu}[g(x)]. 
  \]

\subsubsection*{Convolutions}
The \emph{convolution of functions} $f$ and $g$ is given by 
\[f*g(x) =\mathbb{E}_{y\sim G}[f(y)g(y^{-1}x)].\]
The \emph{convolution of probability measures} $\mu, \nu$ is the probability measure $\mu *\nu$ corresponding to choosing $x\sim \mu, y\sim \nu$ independently and outputting $xy$. It is easy to see that the density function corresponding to $\mu*\nu$ is the convolution of the correponding density functions. 

For a function $f\colon G\to \mathbb{C}$, we denote by $f'$ the function $f'(\sigma) = \overline{f(\sigma^{-1})}$. The adjoint of the convolution with $f$ operator is the operator $g\mapsto f'* g$:
\begin{lem}\label{lem:simple adjointness lemma}
    For all functions $f,g,h\colon G \to \mathbb{C}$, we have 
    \begin{equation}\label{eq:adjointness}
    \langle f*g ,h\rangle = \langle g , f'*h\rangle.
            \end{equation}
\end{lem}
\begin{proof}
For $\sigma \in G,$ let $\Delta_{\sigma}$ be the function whose value is $|G|$ on $\sigma$ and 0 on $G\setminus \{\sigma\}.$ Now $f\mapsto \langle f*g ,h\rangle$ and $f\mapsto \langle g , f'*h\rangle$ are both linear functionals. Hence, it is enough to verify the lemma for the functions $f = \Delta_{\sigma},$ which form a basis. For every function $f$ we have 
$\Delta_{\sigma} * f (\tau)= f(\sigma^{-1}\tau)$. We also have $\Delta_{\sigma}' = \Delta_{\sigma^{-1}}$. Thus, 
\[
\langle \Delta_{\sigma}*g ,h\rangle = \mathbb{E}_{\tau}[g(\sigma^{-1}\tau) h(\tau)]= \mathbb{E}_{\tau}[g(\tau) h(\sigma \tau ] = \langle g, \Delta_{\sigma^{-1}} * h \rangle,
\]
which completes the proof.
\end{proof}

\subsubsection*{The spaces \(W_{\chi}\) of matrix coefficients}
Let $(V,\rho)$ be a complex representation of a finite group $G$, we denote by $V^*$ the dual space of functionals on $V$. Suppose that $\chi = \mathrm{tr} \circ \rho$ is the irreducible character corresponding to $\rho$. Then each function of the form $\sigma \mapsto \varphi(\sigma v)$ with $v\in V$ and $\varphi\in V^*$ is called a \emph{matrix coefficient} for $\chi$. 

The space $W_{\chi}$ spanned by the matrix coefficients of $\chi$ is called the \emph{space of matrix coefficients} for $\chi$. 
The space of functions of $G$ is orthogonally decomposed as a direct sum of the spaces $W_{\chi}.$ 
Thus, every function $f$ can be uniquely othogonally decomposed as $f=\sum_{\chi} f^{=\chi},$ with $f^{=\chi}\in W_{\chi}.$ 

It is a direct consequence of the Peter--Weyl theorem that the functions $f^{=\chi}$ have the following nice formula. 
\begin{fact}\label{fact:peter--weyl}
    For a function \(f\) on \(G\) we have
    \[f^{=\chi}  = \chi(1)f*\chi.\]
\end{fact}

We write $\hat{G}$ for the set of irreducible characters of $G$. The following generalized Frobenius formula is due to Gill, Lifshitz, Pyber, and Szab\'{o}~\cite{GLPS}. 
\begin{lem}\label{lem: GLPS}
    Let $f,g\colon G\to \mathbb{C}$. Then 
    \[\mathbb{E}_{\sigma} \| f^{\sigma} * g \|_2^2 = \sum_{\chi\in \hat {G}} \frac{\|f^{=\chi}\|_2^2 \|g^{=\chi}\|_2^2}{\chi(1)^2}.\]
\end{lem}

\remove{
The classical Frobenius formula is about the special case where $f,g$ take the form $\frac{1_{\sigma^{G}}}{\mu(\sigma^{G})}$ and $\frac{1_{\tau^{G}}}{\mu(\tau^{G})}$. There $f= \sum \hat{f}(\chi) \chi$, where $\hat{f}(\chi) = \langle f ,\chi \rangle = \overline{\chi}(\sigma).$ The functions $f^{=\chi}$ then take the form $\hat{f}(\chi)\chi$. Lemma \ref{lem: GLPS} for class functions together with the orthonormality of the characters, then follows from the classical Frobenius formula, which states that  \[f*g = \sum_{\chi} \frac{\hat{f}(\chi)\hat{g}(\chi)}{\chi(1)}.\]
}

\section{From character theory to growth: the proof of Lemma \ref{lem:from character theory to growth}}

\subsubsection*{Notations}
For a subset $A\subseteq G$ we write $1_A\colon G\to \{0,1\}$ for its indicator. We write $\mu(A)$ for the density of $A$ given by $\frac{|A|}{|G|}$. 
\subsection{The $L_2$-norms of the projections $f^{=\chi}$.}
 We will make use of the following formula for $\|f^{=\chi}\|_2^2.$ 
\begin{lem}\label{lem:formula for f^chi}
Let $G$ be a finite group, $f\colon G\to \mathbb{C}$ be a function. Set the function $f'$ to be given by $f'(\sigma) = \overline{f(\sigma^{-1})}$. Then  
    \[\|f^{=\chi}\|_2^2 = \chi(1)\langle f'*f, \chi \rangle.\]
\end{lem}
\begin{proof}
    By Lemma \ref{lem:simple adjointness lemma} and Fact \ref{fact:peter--weyl} we have 
    \[\chi(1)\langle f'* f, \chi\rangle =  \langle  f, \chi(1) f * \chi\rangle = \langle  f, f^{=\chi} \rangle = \|f^{=\chi}\|_2^2,\]
    where the last equality follows from the definition of $f^{=\chi}$ as the projection of $f$ onto the space $W_{\chi}$ of matrix coefficients for $\chi$. 
\end{proof}

\subsection{Functional version of Lemma \ref{lem:from character theory to growth}}
We now show that a certain technical condition implies that $|A^{\sigma} B|$ is much larger than $B$. The following lemma simplifies to Lemma \ref{lem:from character theory to growth} after a proper choice of the parameters.

\begin{lem}\label{lem:wrapping up}
Let $G$ be a finite group and let $A,B\subseteq G$. Write $f = \frac{|G|}{|A|}1_A$. Let $K>1,\delta>0$ and $S$ be a set of irreducibe characters of $G$ such that for each $\chi \in S$ we have $\|f^{=\chi}\|_2^2\le \frac{\chi(1)^2}{K}$ and  $\sum_{\chi \notin S} \chi(1)^2 \le \frac{\delta|G|}{K|B|}$. Then we have  \[\mathbb{E}_{\sigma}\left[\frac{1}{|A^{\sigma} B|}\right] \le \frac{1+\delta}{K|B|}.\]
In particular, there exists $\sigma\in G$ with \(|A^{\sigma} B| \ge  \frac{K}{1+\delta}|B|.\)
\end{lem}
\begin{proof}
By Lemma \ref{lem:formula for f^chi} we have 
\[
\|f^{=\chi}\|_2^2 = \chi(1) \langle f' * f,\chi \rangle  \le \chi(1) \|f'*f\|_1 \|\chi\|_{\infty} = \chi(1)^2,
\]
and similarly for the density function $g:= \frac{1_B}{\mu(B)}$. Therefore by Lemma \ref{lem: GLPS}, 
    we have 
    \[\mathbb{E}_{\sigma} \|f^{\sigma} * g\|_2^2 = \sum_{\chi}\frac{\|f^{=\chi}\|_2^2 \|g^{=\chi}\|_2^2}{\chi(1)^2} \le \sum_{\chi\in S} \frac{1}{K}\|g^{=\chi}\|_2^2 + \sum_{\chi \notin S} \chi(1)^2.
    \]
    Now 
    \[
    \sum_{\chi\in S} \frac{1}{K}\|g^{=\chi}\|_2^2 \le \frac{\|g\|_2^2}{K} 
    \]
    and 
    \[
    \sum_{\chi \notin S} \chi(1)^2 \le \frac{\delta}{K(\mu(B))} = \delta \|g\|_2^2/K.  
    \]
    This shows that 
    \[
    \mathbb{E}_{\sigma} \|f^{\sigma} * g\|_2^2 \le \frac{(1+\delta)}{K}\|g\|_2^2.
    \]
    Now for $\sigma\in G$ let  $h_{\sigma} = 1_{A\sigma^{-1} B}$. Since the function  $f^{\sigma}$ is the density function of the uniform distribution on $A\sigma^{-1}$, the density function $f^{\sigma}*g$ corresponds to a probability distribution supported on $A\sigma^{-1}B$. Therefore, by Cauchy--Schwarz for all $\sigma$ we have \[1 = \langle f^{\sigma} * g, h_{\sigma}\rangle \le \| f^{\sigma} * g \|_2 \|h_{\sigma}\|_2.\]
    This yields that $\frac{|G|}{|A\sigma^{-1} B|} =  \frac{1}{\|h_{\sigma}\|_2^2} \le \|f^{\sigma}*g\|_2^2.$ Taking expectation over $\sigma$ we have 
\[
       \mathbb{E}_{\sigma \sim G} \frac{|G|}{|A\sigma B|} \le \frac{1+\delta}{K}\|g\|_2^2  = \frac{1+\delta}{K}\frac{|G|}{|B|}.
\]
The first part of the Lemma now follows by rearranging, while noting that $|A\sigma B| = |A^{\sigma}B|$. For the `in particualr' part let $\sigma$ be with $|A\sigma B|$ maximal. Then $\frac{1}{|A\sigma B|} \le \frac{1+\delta}{K|B|}.$ Rearranging, we obtain 
    \[
       |A^{\sigma}B| = |A\sigma B| \ge  \frac{K}{1+\delta}|B|.
    \]
\end{proof}

\subsection{Completing the proof of Lemma \ref{lem:from character theory to growth}}

\begin{proof}[Proof of Lemma \ref{lem:from character theory to growth}]
    Set $\delta = |G|^{-\epsilon/2}.$ By hypothesis we have 
    \[
    \sum_{\chi \notin S}\chi(1)^2 \le |G|^{\epsilon}\le \frac{\delta |G|}{K|B|}.
    \]
    Let $f= \frac{1_{A}}{\mu(A)}$ be the density function for the uniform measure on $A$. Then the function $f'$ is the density function for the uniform meeasure on $A^{-1}$. This shows that $f'*f$ is the density function for the covolution of the measures. 
    By Lemma \ref{lem:formula for f^chi} for every character $\chi$ we  have 
    \[
    \|f^{=\chi}\|_2^2 =    \chi(1) \langle \chi, f'*f \rangle = \chi(1)\overline{\mathbb{E}_{x,y\sim A}[\chi(y^{-1}x)]} \le \frac{\chi(1)^2}{K}.
    \]
    Lemma \ref{lem:wrapping up} now completes the proof. 
\end{proof}

\subsection{A variant for small sets $A$}

For small sets $A$ we make use of the following lemma in place of Lemma \ref{lem:from character theory to growth}.

\begin{lem}\label{lem: general LNS}
There exists $c>0$, such that the following holds. Let $\epsilon \in(0,1/4), A$ be a subset of a finite group $G$, and suppose that there exists  $\sigma \in A^{-1}A$ with $|\chi(g)| \le (1-\epsilon) \chi(1)$ for all irreducible character $\chi \ne 1.$ 
Then for each $B\subseteq G$ with $\mu(B) < c\epsilon,$ there exists $\sigma\in G$ with $|A^{\sigma} B| \ge (1 - \epsilon)^{c} |B|$.  
\end{lem} 
\begin{proof}
Let $a,b\in A$ be with $\chi(b^{-1}a)\le (1-\epsilon)\chi(1)$ for all $\chi$. Then without loss of generality $A = \{a,b\}.$ Set $f = 1_A/\mu(A)$. By Lemma \ref{lem:formula for f^chi} we have 
\[\|f^{=\chi}\|_2^2 = \chi(1) \langle f' * f,\chi \rangle \le 3/4\chi(1)^2 + 1/4 \chi(1)^2(1-\epsilon) = (1-\epsilon/4) \chi(1)^2.\] 

By Lemma \ref{lem:wrapping up} with $K = (1-\epsilon/4)$ and $\frac{1}{K\mu(B)}$ in the role $\delta$ we have 
    \[
    |A^{\sigma} B| \ge |B| \frac{K}{(1 +  K\mu(B))} \ge \frac{(1-\epsilon/4)}{1 + 2c\epsilon} \ge  (1 - \epsilon)^{c}|B|,
    \]
    provided that $c$ is sufficiently small.
\end{proof}

\section{Fourier anti-concentration inequalitites}

In this section our goal is to prove upper bounds on $\|f^{=\chi}\|_2^2$ for the density function $f= \frac{1_A}{\mu(A)}$, where $A$ is a subset of a classical finite simple group. By Lemma \ref{lem:formula for f^chi} our goal can be alternatively described as giving upper bounds for $\mathbb{E}_{x,y\sim A}[\chi(y^{-1}x)].$ We achieve this via the Larsen--Tiep~\cite{larsen2024uniform} character bounds. 
\subsection{Anti concentration}

\subsubsection*{Bounds for inner products with characters} 
Theorem \ref{thm:larsen--tiep} of Larsen and Tiep upper bounds the character values. We now deduce from it a bound on the inner product of functions with characters.  
Let $f= \frac{1_A}{\mu(A)}.$ Then $\|f\|_{\infty} = 1/\mu(A)$ and rearranging, we obtain $|A| = \frac{|G|}{\|f\|_{\infty}}$. Moreover, when $A$ is a conjugacy class $\sigma^{G},$ then $\langle f,\chi \rangle =\overline{\chi}(\sigma).$

Theorem \ref{thm:larsen--tiep} can be used to prove an analogue upper bound about the inner product between a function and a character, where $s = \frac{|G|}{\|f\|_{\infty}}$ takes the role of the size of $|A|$.  

\begin{lem}\label{lem:character bounds for general functions}
There exists $c>0,$ such that the following holds. 
    Let $G = Cl_{n}(\mathbb{F}_q)$ be a classical finite simple group. Let $f$ be a density function, and let $\alpha>0$. Write $s = \frac{|G|}{\|f\|_{\infty}}$. Then for every character $\chi$
    \begin{equation}\label{eq:larsen--tiep class functions}
    |\langle f ,\chi \rangle| \le  \chi(1)^{1- \frac{c \alpha \log s}{\log |G|}} + \frac{\chi(1)}{s^{1-\alpha}}.
    \end{equation}
\end{lem}
\begin{proof}
First we set up the dependency of the constants on each other. Let $c_3$ be the constant $c$ of Fact \ref{fact:sizes of conjugacy classes}, $C_1$ be the constant $C$ of Fact \ref{fact:number of conjugacy classes}, $c_2$ be the constant $c$ of Theorem \ref{thm:larsen--tiep}, and  we let $c_1>0$ be sufficiently small with respect to the constants $c_3,C_1$.  

Let $A$ be the set of conjugacy classed of size $>s^{c_1\alpha}$ and $B$ be its complement consisting of conjugacy classes of size $\le s^{c_1\alpha}$. We split the cotribution to the inner product in the left hand side of \eqref{eq:larsen--tiep class functions} according to the contribution from conjugacy classes in $A$ and from $B$.
By the Larsen--Tiep Theorem (Theorem \ref{thm:larsen--tiep}) for all $x\in A$ we have \[|\chi(x)| \le \chi(1)^{1-c_2\frac{\log |x^{G}|}{\log |G|}}\le \chi(1)^{1- c_2 \frac{\log s^{\alpha c_1}}{\log |G|}} = \chi(1)^{1-  \alpha c_1 c_2 \frac{\log s}{\log |G|}}.\] 
Now for $x\in B$ we have 
$|\chi(x)|\le \chi(1)$.
Hence, \[|\langle f ,\chi \rangle| = \left|\mathbb{E}[f(x)\chi(x)1_{A}] + \mathbb{E}[f(x)\chi(x)1_{B}]\right| \le \chi(1)^{1- c_1 c_2\alpha \frac{\log s}{\log |G|}}  + |\mathbb{E}[f 1_{B}]| \chi(1),\]  
where we used the fact that $\mathbb{E}[f1_A]\le \mathbb{E}[f]\le 1.$

Setting $c = c_1c_2$, it remains to show that  $|\mathbb{E}[f 1_{B}]| \le s^{\alpha -1}.$
We have 
\[\mathbb{E}[f 1_{B}] \le \|f\|_{\infty} \|1_{B}\|_1 = \frac{|G|}{s} \|1_{B}\|_1. \] We may now make use of the fact that every conjugacy class of $G$ has size $\ge q^{c_3n}$ (Fact \ref{fact:sizes of conjugacy classes}) and the number of conjugacy classes of $G$ is $\le q^{C_1 n}$ (Fact \ref{fact:number of conjugacy classes}). 
In the case where  $s^{c_1 \alpha}\ge q^{c_3 n}$ we have 
\[\|1_{B}\|_2^2 = \sum_{g^{G}\subseteq B} \frac{|g^{G}|}{|G|}\le q^{C_1 n} \frac{s^{c_1 \alpha}}{|G|} \le \frac{s^{c_1\alpha(1 + C_1/c_3)}}{|G|}.
\]

Setting $c_1 = \frac{1}{1 + C_1/c_3}$ completes the proof of the case $s^{c_1 \alpha}>q^{c_3 n}.$

In the remaining case where $s^{c_1\alpha } < q^{c_3 n}$ we have $B = \{1\}$.
Thus,
\[
\mathbb{E}[f 1_{B}] = \frac{f(1)}{|G|} \le \frac{1}{s}\le s^{\alpha - 1}. 
\]
\end{proof}

\begin{thm}\label{thm: Larsen--Tiep corollary}
There exists $c>0$, such that the following holds. Let $\epsilon >0$, and let  $A$ be a subset of a classical finite simple group $G$, with $|A|\ge 16^{1/c\epsilon}$ and $f=\frac{|G|}{|A|}1_A$. Then for every character $\chi$ of dimension $\ge |G|^{\epsilon}$ we have
\[
\mathbb{E}_{x,y\sim A}[\chi(y^{-1}x)] \le \frac{\chi(1)}{|A|^{c\epsilon}}.
\]
Consequently,
    \[
    \|f^{=\chi}\|_2^2 \le \frac{\chi(1)^{2}}{|A|^{c\epsilon}}.
    \]
\end{thm}
\begin{proof}
Set $f = \frac{|G|}{A}1_A$,  and $g = f'*f,$ where  $f'= \frac{|G|}{A}1_{A^{-1}}.$ Then By Young's convolution inequality,  $\|g\|_{\infty} \le \|f'\|_2\|f\|_2 = \frac{|G|}{|A|}.$
Let $c'$ be the constant $c$ of Lemma~\ref{lem:character bounds for general functions}, and set $c = \min(c'/4,1/100)$. We may now plug in Lemma \ref{lem:character bounds for general functions} with $\alpha =1/2$ to obtain  that 
\[
|\langle g, \chi \rangle| \le \chi(1)^{1- 2c \frac{ \log |A|}{\log |G|}} + \frac{\chi(1)}{|A|^{1/2}} \le \chi(1) (|A|^{-2c \epsilon} + |A|^{-1/2}) \le \chi(1)|A|^{- c \epsilon},
\]

Now 
\[
\langle g, \chi \rangle  = \mathbb{E}_{y,x\sim A}[\overline{\chi}(y^{-1}x)] \frac{\|f^{=\overline{\chi}}\|_2^2}{\chi(1)}  
\]
by Lemma \ref{lem:formula for f^chi}.
The theorem follows by rearranging. 
\end{proof}

The following theorem is better suited for the low rank regime. 

\begin{corollary}\label{cor: Gluck}
There exists $c>0$, such that the following holds. Let $\epsilon >0$, and let $A\subseteq G$ be a subset of a finite simple group of Lie type over $\mathbb{F}_q$, and set $f=\frac{|G|}{|A|}1_A$. Then for every nontrivial character $\chi$ we have
\[
\mathbb{E}_{x,y\sim A}[\chi(y^{-1}x)] \le \frac{\chi(1)}{q^{c}} + \frac{1}{|A|}.
\]
Consequently,
    \[
    \|f^{=\chi}\|_2^2 \le \chi(1)^{2}\left({q^{-c} + \frac{1}{|A|}}\right).
    \]
\end{corollary}
\begin{proof}
    This follows from Theorem \ref{thm:Gluck} as $y^{-1}x = 1$ with probability $\frac{1}{|A|}$ and by Theorem~\ref{thm:Gluck} we have $|\chi(\sigma)|\le \chi(1)q^{-c}$ for an absolute constant $c>0$ for $\sigma\ne 1.$ 
\end{proof}
\remove{
\begin{thm}\label{thm: when |A| is small}
There exists $c>0$, such that the following holds. Let $A$ be a subset of a classical finite simple group $G$ and $f=\frac{|G|}{|A|}1_A.$ Then for every character $\chi$ we have 
    $\|f^{=\chi}\|_2^2 \le \chi(1)^{2- \frac{c}{n}} + \frac{\chi(1)^2}{|A|}.$
\end{thm}
\begin{proof}
    Set $f = \frac{|G|}{A}1_A$ and $f'= \frac{|G|}{A}1_{A^{-1}},$ and $g = f'*f.$ Then we have $\|g\|_{\infty} \le \|f'\|_2\|f\|_2 = \frac{|G|}{|A|}.$
We may now plug in Lemma \ref{lem:When s is small} with $\alpha =1/2$ to obtain that 
\[
|\langle g, \chi \rangle| \le \chi(1)^{1- c/n} + \frac{\chi(1)}{|A|}.
\]
Now 
\[
\langle g, \chi \rangle = \frac{\|f^{=\chi}\|_2^2}{\chi(1)}   
\]
by Lemma \ref{lem:formula for f^chi}.
The theorem follows by rearranging.
\end{proof}
}

\section{Completing the proofs for finite simple groups of Lie type}
In this section we prove Theorem \ref{thm:skew product theorem_intro} and then deduce the Liebeck--Nikolov--Shalev conjecture for the finite classical groups from it. 
\remove{
\begin{thm}
 For each $\epsilon \in(0,1/4)$ there exists $\delta>0$, such that the following holds. Let $G$ be a classical finite simple group and let $A\subseteq G$.
Then for each $B\subseteq G$ of size $<\delta |G|$ there exists $\sigma\in G$ with $|A\sigma B| \ge (1 + \epsilon/10) |B|$.  
\end{thm} 
\begin{proof}
    When combining Fact \ref{fact:sizes of conjugacy classes}, which lower bounds the sizes of conjugacy classes by $q^{cn}$, Fact \ref{fact:Quasirnadomness of simple classical groups}, which yields $\chi(1) \ge q^{cn}$ for each character $\chi \ne 1$ when combining these with the fact that $|G|\le q^{Cn^2}$ for an absolute constant $C>0$ with Theorem \ref{thm:larsen--tiep} we obtain that $\chi(g)\le q^{-c}$ for an absolute constant $c>0$ and every $g\ne 1.$ The Theorem now follows from Lemma \ref{lem: general LNS}. 
\end{proof}
}

We start with a lower bound of the form  $|A^{\sigma} B|\ge (1+c)|B|$ for an absolute constant $c>0$. This bound yields Theorem \ref{thm:skew product theorem} when $|A|^{\epsilon} = O(1)$.

\begin{thm}\label{thm: general LNS}
There exists $c>0$, such that the following holds. 
    Let $G$ be a finite simple group of Lie type over $\mathbb{F}_q$, let $A$ have size $\ge 2$, and let $B\subseteq G$ have size $<c|G|.$ Then there exists $\sigma\in G$ with $|A^{\sigma} B| \ge (1 + c)|B|$.  
\end{thm}
\begin{proof}
 By Lemma \ref{lem:When s is small} there exists an absolute constant $c'>0$, such that $\chi(\sigma)\le q^{-c'}$ for every irreducible character $\chi \ne 1$ and $\sigma \ne 1$. Lemma \ref{lem: general LNS} now completes the proof. 
\end{proof}

\begin{lem}\label{lem:bounded rank special case}
For every $r>0$ there exists $c>0$, such that the following holds. Let $G$ be a finite simple group of Lie type over $\mathbb{F}_q$ of rank at most $r$, and let $\epsilon>0.$. Then for every subsets $A,B\subseteq G$, with $|B|\le \min(|G|^{1-\epsilon},c|G|)$ there exists $\sigma \in G$ with $|A^{\sigma}B|\ge |A|^{c\epsilon}|B|.$
\end{lem}
\begin{proof}
 Let $c'$ be as in Corollary \ref{cor: Gluck}, $c$ sufficiently small with respect to $r,c'$, $K = |A|^{-2c\epsilon}$ and $\delta = \frac{|B|}{K|G|}$.

 By Theorem \ref{thm: general LNS} we may assume that $|A|^{\epsilon}$ is sufficiently large with respect to $r,c'$. 
 Now as $|A|\le |G| \le q^{Cr^2}$ for an absolute constant $C>0$ we have $\frac{1}{|A|} + q^{-c'} \ge |A|^{-2c\epsilon} =K$, provided that $|A|$ is sufficiently large, and $c$ is sufficiently small with respect to $c',r$. We may now apply Corollary \ref{cor: Gluck} and Lemma \ref{lem:wrapping up} to obtain that $|A^{\sigma}B| \ge \frac{K|B|}{1+\delta}$ for some $\sigma\in G.$ Now $\delta \le \frac{|G|^{\epsilon}}{K} \le 1,$ provided that $c\le 1.$ So $|A^{\sigma}B|\ge K|B|/2\ge |A|^{c\epsilon}|B|$, provided that $|A|^{\epsilon}$ is sufficiently large. 
\end{proof}

 We now prove Theorem \ref{thm:skew product theorem}, which we restate for the convenience of the reader. 
\begin{repskewthm}
There exist $c>0$, such that the following holds. Let $G$ be a finite simple group of Lie type, let $\epsilon >0$, and suppose that sets $A,B\subseteq G$ satisfy $|B|<\min(|G|^{1 -\epsilon}, c|G|).$ Then 
    \[
    |A^\sigma B| \ge |A|^{c\epsilon}|B|
    \]
    for some $\sigma \in G.$
\end{repskewthm}
\begin{proof}

By Lemma \ref{lem:bounded rank special case} the statement holds in the bounded rank regime. 
 We may therefore assume that $G= Cl_n(\mathbb{F}_q)$ is a classical finite simple group. Let $c'$ be sufficiently small, let $c = c(c')$ be sufficiently small with respect to $c'$, and 
    let $S$ be the set of characters of dimension $\ge |G|^{c'\epsilon}.$

Suppose first that $|A|>32^{1/c'\epsilon}.$
    By Theorem \ref{thm: Counting low dimensional representations} we have 
    \[
    \sum_{\chi \notin S} \chi(1)^2 \le |G|^{\epsilon /4}.
    \]
    Moreover, by Theorem \ref{thm: Larsen--Tiep corollary} we have $|\mathbb{E}_{x,y\sim A}[\chi(y^{-1}x)]| \le \frac{\chi(1)}{|A|^{2c'\epsilon}}$ for every character $\chi \in S$ provided that $c'>0$ is sufficiently small. 
    By Lemma \ref{lem:from character theory to growth} we now obtain that 
    there exists $\sigma \in G$ with $|A^{\sigma} B|\ge \frac{|A|^{2c'\epsilon}|B|}{2} \ge |A|^{c'\epsilon}|B|.$ This completes the proof of the lemma for $|A|>32^{\frac{1}{c'\epsilon}}$, provided that $c\le c'$.

        Suppose now that $|A|\le 32^{\frac{1}{c'\epsilon}}.$ Then by
        Theorem \ref{thm: general LNS} we have $|A^{\sigma} B|\ge (1+c'')|B|$ for some absolute constant $c''>0.$ This completes the proof, as provided that $c$ is sufficiently small, we have $|A|^{c\epsilon}\le (1+c'')$.
\end{proof}

We are now ready to prove the Liebeck--Nikolov--Shalev conjecture for finite simple groups of Lie type.

\begin{thm}\label{thm: Lie type LNS}
    Let $G=Cl_n(\mathbb{F}_q)$. Then for every set $A$ and integer $i$ there exist $\sigma_1,\ldots, \sigma_i\in G$ with $|A^{\sigma_1}A^{\sigma_2}\cdots A^{\sigma_i}|\ge \min(A^{ci}, |G|).$
\end{thm}
\begin{proof}
    By applying Theorem \ref{thm:skew product theorem} iteratively we may sequentially find $\sigma_1,\ldots, \sigma_{m-1}$  an absolute constant, such that  \[|A^{\sigma_{m-1}}A^{\sigma_{m-2}} \ldots A|\ge |A|^{c'}|A^{\sigma_{m-2}} \ldots A|\ge \cdots \ge  |A|^{c'(m-1) + 1} \ge |A|^{c'm}\] for some absolute constant $c'>0$ until $|A^{\sigma_{m-1}} \ldots A^{\sigma_{1}}A|\ge |G|^{1/2}$. Let $m$ be the smallest such that \[|A^{\sigma_{m-1}} \ldots A^{\sigma_1}A|\ge |G|^{1/2}.\]
    
     By Theorem \ref{thm: GLPS main theorem} there exists an absolute contant $C>0$, such that there are $C$ conjugates of the set $A^{\sigma_{m-1}} \ldots A^{\sigma_1}A$ whose product is $G$. Let $c = c'/C.$ Then for $i\le Ci$ there are $i/M$ conjugates of $A$ whose product has size $\ge |A|^{c' i/m} = |A|^{ci}$. Now for $i\ge Cm$ there are $i$ conjugates of $A$ whose product is $G$. This completes the proof of the theorem. 
\end{proof}

\section{Alternating groups}

We now move on to proving the Liebeck--Nikolov--Shalev theorem for alternating groups.
\subsection{Growth or small support} 
In this section we prove the following lemma that yields a skew product theorem for $A$ of the form $|A^{\sigma} B|> |A|^{c}|B|$ for some absolute constant $c>0$, unless a shift $a^{-1}A$ contains various elements of small support. Our proof makes use of Lemma \ref{lem:from character theory to growth} and the character bounds in Theorem \ref{thm:Schlage-Puchta A_n}.

\begin{lem}\label{lem: growth or small support}
    For every $\epsilon>0$ there exists $c>0$, such that the following holds. Let $G=A_n$ be the alternating group, and let $A\subseteq G$. Then either there exists $a\in A$, such that the set $a^{-1}A$ contains at least $|A|^{0.9}$ elements of support size at most $\epsilon \log |A|$, or for all sets $B\subseteq G$ with $|B|<|G|^{1-\epsilon}$, there exists $\sigma \in G$ with 
    \[
    |A^{\sigma} B|\ge |A|^c |B|.
    \]
\end{lem}
\begin{proof}
Write $G = A_n$.
By decreasing $c$ if necessary we may assume that $n$ is sufficiently large as a function of $\epsilon$. 
Suppose that $a^{-1}A$ contains $< |A|^{0.9}$ elements of support size $m$ for each $a\in A$. We show that  $|A^{\sigma} B|\ge |A|^c|B|$ for some $\sigma \in G$. 

Let $S$ be the set of characters $\chi$ of dimension $\ge |G|^{\epsilon/10}$.  Then $\sum_{\chi\notin S} \chi(1)^2 \le |G|^{\epsilon/2},$ provided that $n$ is sufficiently large, by Lemma \ref{lem:Liebeck-shalev zeta}.

Now let $\chi \in S$. Choose $x,y\sim A$ uniformly and independently. 
 Let $E$ be the event that $y^{-1}x$ has support size $<m.$ Then by hypothesis we have \[\Pr[ E ] < \frac{|A|^{0.9}}{|A|}=|A|^{-0.1}.\] When the event $E$ is satisfied we may simply use the trivial bound $|\chi(y^{-1}x)|\le \chi(1).$ On the other hand, whenever the event $E$ is not satisfies and $y^{-1}x$ has support $\ge m$ we may apply Theorem \ref{thm:Schlage-Puchta A_n} to obtain the improved bound 
 \[
 |\chi(y^{-1}x)|\le \chi(1)^{1- c'\frac{\epsilon\log |A|}{\log |G|}}
 \]
 for an absolute constant $c'>0$. Thus, 
 \[
 \frac{|\chi(y^{-1}x)|}{\chi(1)} \le |G|^{-\epsilon c'/10 \frac{\log |A|}{\log |G|}} =|A|^{-c'\epsilon /10}.
 \]
    By Lemma \ref{lem:from character theory to growth} (with $K = |A|^{c'\epsilon/10}$) for every set $B$ of size $<|G|^{1-\epsilon},$ there exists $\sigma \in G$ with  
    \[
    |A^{\sigma} B|\ge \frac{|B||A|^{c'\epsilon/10}}{1 + |G|^{-\epsilon/4}} \ge |A|^{c}|B|,
    \]
    provided that $c$ is sufficiently small as a function of $\epsilon,c',$ which completes the proof of the lemma.   
\end{proof}

\subsection{Growth for sets of elements of small support} 
In this section we obtain growth of the form $|A^{\sigma}B|\ge |A|^{c}|B|$ for sets of elements of small support. 

Note that $\sigma^{A_n} = \sigma^{S_n}$ if and only if the centralizer of $\sigma$ contains an odd permutation. In particular, for elements $\sigma$ of support size $\le n-2$, we have $\sigma^{A_n} =\sigma^{S_n}$, as such a permutation $\sigma$ has a transposition in its centralizer. 
Let $\sigma, \tau\in A_n$ be permutations having disjoint support, while having their support sizes sum up to at most $n-2$, then we set the \emph{direct sum of their conjugacy classes} to be given by \[\sigma^{A_n} \oplus \tau^{A_n} := (\sigma \tau)^{A_n}.\]

We now show that for subsets of conjugacy classes $A\subseteq \sigma_1^{A_n}$, $B\subseteq \sigma_2^{A_n},$ there exists $\sigma$, such that $A^{\sigma}B$ contains various elements in the conjugacy class  $\sigma_1^{A_n}\oplus \sigma_2^{A_n}$. 
We establish that by a second moment argument.

\begin{lem}\label{lem: skew product theorem for elements of small support}
There exists $C>0$, such that the following holds. Let $m,r\le n/4$, let $\sigma_1 \in A_n$ be of support $m$ and $\sigma_2 \in A_n$ be of support $r$. Suppose that $A\subseteq \sigma_1^{A_n}$ satisfies $|A| \ge C^m$ and that $B\subseteq \sigma_2^{A_n}$. Then \[\left|A^{\sigma} B \cap \left(\sigma_1^{A_n}\oplus \sigma_2^{A_n} \right)\right| \ge |A|^{1/3}|B|\] for some $\sigma \in A_n.$ 
\end{lem}
\begin{proof}
    Choose randomly and independently $\tau,\pi \sim A_n$. 
    For each $a\in A,b\in B$ and $x\in \sigma_1^{A_n}\oplus \sigma_2^{A_n}$ set $X_{a,b,x}$ be the indicator of the event $a^{\tau} b^{\pi} = x.$
    We write $X_x = \sum_{a\in A,b\in B} X_{a,b,x}$
    and let $E_x$ be the event that $X_x>0.$ In words, $E_x$ is the indicator for the event $x\in A^{\tau}B^{\pi}$ and $X_{x}$ counts how many times $x$ can be represented as a product of elements in $A^{\tau}$ and $B^{\pi}.$  
    
    We will show the inequality 
    \[\mathbb{E}_{\tau,\pi}[|A^{\tau}B^{\pi} \cap \sigma_1^{A_n}\oplus \sigma_2^{A_n}|] = \sum_{x\in \sigma_1^{A_n}\oplus \sigma_2^{A_n}} \Pr_{\tau, \pi}[E_x] > |A|^{1/3}|B|.\]
    By the first moment method this will imply that for some $\tau,\pi \in A_n$ we have \[| A^{\tau} B^{\pi} \cap \sigma_1^{A_n}\oplus \sigma_2^{A_n}| \ge |A|^{1/3}|B|.\] This will complete the proof with $\sigma =\tau \pi^{-1}$, as conjugation by $\pi^{-1}$ preserves sizes.
    
    Fix $x\in \sigma_1^{A_n}\oplus \sigma_2^{A_n}.$ Our proof will be complete once we show that \[\Pr_{\tau, \pi}[E_x] > \frac{|A|^{1/3}|B|}{|\sigma_1^{A_n}\oplus \sigma_2^{A_n}|}.\]
    To accomplish that we lower bound $\Pr[E_x]$ via the Payley--Zigmund inequality 
    \[
        \Pr[E_x] \ge \frac{\mathbb{E}^2[X_x]}{\mathbb{E}[X^2_x].}
    \]    
    Let us start by computing \[\mathbb{E}[X_x] =\sum_{a\in A,b\in B} \mathbb{E}[X_{x,a,b}].\] Let $a\in A,b\in B$. By symmetry, once conditioning that $a^{\sigma}$ and $b^{\tau}$ have disjoint supports the distribution of $ a^{\sigma}b^{\tau}$ is uniform in the conjugacy class $\sigma_1^{A_n}\oplus \sigma_2^{A_n}.$ Now as the supports of $a^{\sigma}$ and $b^{\tau}$ are evenly distributed in $\binom{[n]}{m}$ and $\binom{[n]}{r}$ respectively, the probability that the supports are disjoint is $\frac{\binom{n-r}{m}}{\binom{n}{m}}$. We therefore have \[\mathbb{E}[X_{a,b,x}] = \frac{\binom{n-r}{m}}{\binom{n}{m}|\sigma_1^{A_n}\oplus \sigma_2^{A_n}|}.\]
    By linearity of expectation we obtain that \[\mathbb{E}[X_x] = \frac{\binom{n-r}{m}|A||B|}{\binom{n}{m}|\sigma_1^{A_n}\oplus \sigma_2^{A_n}|}.\]

    We now move on to upper bounding 
    $\mathbb{E}[X_x^2].$
    For $a\in A$ let 
    \[Y_{a,x} = \sum_{b\in B}X_{a,b,x}\] The random variable $Y_{a,x}$ is the indicator for the event $x\in a^{\tau}B^{\pi}$.
    By linearity of expectations we have 
    \[
    \mathbb{E}[X_x^2] = \sum_{a,a'\in A,b\in B} \mathbb{E}[X_{a,b,x}Y_{a',x}]
    \]
    Let $a,a'\in A,b\in B$. The number of coordinates $i$ for which $a(i)\ne a'(i)$ is always an even number. Let us denote this number by $2\ell$.  We define a partial ordering $\le$, where  $\tau_1 \le \tau_2$ if the support of $\tau_1$ is contained in the support of $\tau_2$, and $\tau_1$ and $\tau_2$ agree on the support of $\tau_1$.  
    
    The random variable $Y_{a',x}$, which takes either the value 1 or 0, can take the value 1 only if $ a'{^\sigma} \le x$. We may therefore upper bound $\mathbb{E}[X_{a,b,x}Y_{a',x}]$ 
    by the probability that $a^{\sigma}  b^{\tau} = x$ times the conditional probability that $a '{^\sigma} \le x$
    given that $a^{\sigma}b^{\tau} = x$. 
    In Claim \ref{claim: technical upper bound on the probability} we show that this happens with probability at most 
    \(\left(\frac{1}{n-2m}\right)^{\ell/2}\) when $a'$ agrees with $a$ on the intersection of their support and 0 otherwise.

  Let $S_{\ell}$ be the set of such elements $a'$ in $a^{A_n}$
  that disagree with $a$ on $2\ell$ coordinates while agreeing with $a$ on the intersection of their supports. Then the size of the set $S_{\ell}$ is at most $\ell! \binom{n-m}{\ell}\binom{m}{\ell} \le n^{\ell} 2^m.$ Moreover, $|S_{\ell}| \le |A|.$ 
  Combining the bounds by taking the geometric mean we obtain that \[|S_{\ell}|  \le |A|^{1/2} (n^{\ell}2^{m})^{1/2}.\]
  Hence,
  summing over the set $S_{\ell}$ of elements in $A$ that disagree with $a$ on exactly $2\ell$ coordinates we obtain two bounds: 
  The bound 
  \begin{align*}
  \sum_{a'\in S_{\ell}} \mathbb{E}[X_{a,b,x}Y_{a',x}] & \le \sum_{a'\in S_{\ell}}(n-2m)^{-\ell/2}\mathbb{E}[X_{a,b,x}] \\ & = (n-2m)^{-\ell/2} |S_{\ell}|\mathbb{E}[X_{a,b,x}] \\& \le |A|^{1/2} 2^{1.5m} \mathbb{E}[X_{a,b,x}].
  \end{align*}
  Summing this bound over all $\ell,a,b$ 
  we obtain \[\mathbb{E}[X_{x}^2] \le m2^{1.5m}|A|^{1/2} \mathbb{E}[X_x] \le |A|^{2/3}\mathbb{E}[X_x],\]
  provided that $C$ is sufficiently large.
  Applying the Payley--Zigmund inequality while plugging in our computation for $\mathbb{E}[X_x]$ yields that \[\Pr[E_x]\ge \frac{|A|^{1/3}|B|}{|\sigma_1^{A_n}\oplus \sigma_2^{A_n}|}.\]
  Summing over all $x\in \sigma_1^{A_n}\oplus \sigma_2^{A_n}$ completes the proof. 
\end{proof}

We now complete the proof of the lemma by filling out the proof of the missing claim.
\begin{claim}\label{claim: technical upper bound on the probability}
In the setting of the proof of Lemma \ref{lem: skew product theorem for elements of small support}, for every $x$ the probability that $a'{^{\sigma}}\le x$ given that $a^{\sigma}b^{\tau} = x$ is at most
    \[\left(\frac{1}{n-2m}\right)^{\ell/2}\] when 
    $a'$ and $a$ agree on the intersection of their support, and 0 otherwise. 
\end{claim}
\begin{proof}
    In fact, we prove a stronger upper bound on the conditional probability that $a'{^{\sigma}}\le x$, given an arbitrary value of $a^{\sigma}$, $\tau.$

    Since the permutation $\sigma,\tau$ are independent, the value of $\tau$ is irrelevant. Moreover, the distribution of $\sigma$ given $a_{\sigma}$ is evenly distributed in a coset $C_{a}\sigma_0$, where $C_{a}$ is the cenralizer of $a$. Now the centralizer $C_{a}$ contains the subgroup $H$ of all even permutations fixing the support of $a$ and we may condition furthor on the coset of $H$ that $\sigma$ belongs to. 
    
    To conclude, it suffices to prove that for every $\sigma'$ with $a^{\sigma'}\le x$, the  upper bound 
\(\Pr_{\sigma \sim H \sigma'}[a'{^\sigma} \le x]\) by \(\left(\frac{1}{n-2m}\right)^{\ell/2}\) holds when $a'$ and $a$ agree on the intersection of their support and that we never have $a'^{\sigma}\le x$ if $a'$ and $a$ do not agree on the intersection of their support. 

    It is enough to prove the lemma with $a$ and $a'$ replaced by  $a^{\sigma'}$, and $a'{^{\sigma'}},$ and with $\sigma'$ replaced by 1. We may therefore, assume without loss of generality $\sigma' = 1$ and therefore $a \le x.$
    
    It is now evident that $a'{^\sigma} \le x$ if and only if $a'$ agrees with $a$ on the support of $a$ and $\sigma$ sends each (non fixed point) cycle of $a'$ outside the support of $a$ to a cycle of $x$. In particular if $a',a$ disagree on the intersection of their support we never have $a'{^\sigma} \le x$.
    
    Suppose now that $a,a'$ agree on the intersection of their support. Ordering the cycles of $a'$ arbitrarily, the conditional probability that the $i$th $k_i$-cycle of $a'$ is sent to a cycle of $x$ given that this occured for the previous cycles is at most $1/(n-2m)^{k_i-1}$. As the sum $\sum_i(k_i-1) \ge \ell/2$. The probability that $a'{^{\sigma}}\le x$ is at most 
    \(\left(\frac{1}{n-2m}\right)^{\ell/2}.\)  This completes the proof of the claim. 
    \end{proof}
    
\subsection{Growth up to $e^{Cn}$}
In this section or goal is to show that for all sets $A_1,\ldots, A_i$
there exists $\sigma_1,\ldots \sigma_i\in G$ with \[|A_1^{\sigma_1} \cdots A_i^{\sigma_i}|\ge \min(|A_1|^{c}\cdots |A_i|^c, e^{Cn})
\]
for an arbitrarily large constant $C>0$ and a sufficiently small constant $c>0$.

Our idea is to combine 
Lemma \ref{lem: growth or small support} that implies that we have a skew product theorem for sets $A,B$ unless a large subset $A'\subseteq A$ is contained in a translate of a conjugacy class of permutations of small support. Lemma \ref{lem: skew product theorem for elements of small support} then complements it by implying growth for sets of elements of small support. 

In order to carry this plan out we prove the following lemma that reduces us to the case that all the sets $A_i$ are contained in a conjugacy class of small support. It is stated in greater generality for future use. 

\begin{lem}\label{lem: getting rid of the easy either case in either or situations.} 
    Let $\ell< m$ be integers, let $c,c'\in (0,1)$, and let  $\mathcal{E}$ be a collection of subsets of a finite group $G$. Suppose that for each $A_1,\ldots, A_i\in \mathcal{E}$ of size $\ge \ell^{c'}$ there exist $\sigma_1,\ldots, \sigma_i\in G$ with 
    
    \[|A_1^{\sigma_1} \cdots A_i^{\sigma_i}|\ge \min(m, |A_1|^{c}\cdots |A_i|^{c}).\]
    Suppose additionally, that for each subsets $A,B\subseteq G$, with $|A|\ge \ell,|B|\le m$ we either have 
    \[|A^{\sigma}B|>|A|^{c}|B|\] for some $\sigma\in G$ or there exists $\sigma \in G$ and $A'\in \mathcal{E}$ with $\sigma A' \subseteq A$ and $|A'|>|A|^{c'}.$ Then for all sets 
    $A_1,\ldots, A_i\subseteq G$, of size at least $\ell$, there exists $\sigma_1,\ldots, \sigma_i$ with 
    \[
        |A_1^{\sigma_1} \cdots A_i^{\sigma_i}|\ge \min(m, |A_1|^{cc'/2}\cdots |A_i|^{cc'/2})
    \]
\end{lem}
\begin{proof} 

Let $A_1,\ldots A_i\subseteq G$. Our goal is to show the existence of  $\sigma_1,\ldots, \sigma_i$ with 
    \[
        |A_1^{\sigma_1} \cdots A_i^{\sigma_i}|\ge \min(m, |A_1|^{cc'/2}\cdots |A_i|^{cc'/2})
    \]

Now this holds for the sets $A_i$ if and only if it holds for translates of the sets $A_i$. Thus, without loss of generality we may assume that $ 1\in A_i$ for all $i$. Let $S$ be the set of coordinates $i$ with $A_i\in \mathcal{E}$ and let $T$ be its complement. 

Assuming, as we may, that $1\in A_i$ for all $i$ we have 
\[
    |A_1^{\sigma_1}A_2^{\sigma_2} \cdots A_i^{\sigma_i}|\ge |\prod_{i\in S}A_i^{\sigma_i}|,
\]
for all choices of $\sigma_1,\ldots, \sigma_i$,
(where if $S = \{i_1,\ldots, i_r\}$ with $i_1<i_2<\cdots <i_r$, then  $\prod_{i\in S}A_i$ denotes $A_{i_1}A_{i_2}\cdots A_{i_r}).$ Let $\tau_i A_i'\subseteq A_i$ be with $A_i'\in \mathcal{E}$ of size $\ge |A_i|^{c'}.$ Then by hypothesis there exist 
$\{\sigma_i'\}_{i\in T}$, which then correspond to \( \{\sigma_i \}_{i\in T},\)  with 
\begin{equation}\label{eq:prod i in s}
\min(m, |A_1|^{cc'}\cdots |A_i|^{cc'})
 \le \left|\prod_{i\in T} A'_{i}{^{\sigma'_i}}\right|
= \left| \prod_{i\in S}(\tau_i A_i')^{\sigma_i} \right|
\le \left| \prod_{i\in S}(A_i)^{\sigma_i} \right|, 
\end{equation}
where the equality used the fact that a product of translates of certain sets can be alternatively expressed as a translate of the product of conjugates of the sets, which therefore has the same size as the product of the conjugates of the sets, with the translation ommited. 

Similarly,  
\[
    |A_1^{\sigma_1}A_2^{\sigma_2} \cdots A_i^{\sigma_i}|\ge |\prod_{i\in T}A_i^{\sigma_i}|.
\]
for every choice of $\sigma_1,\ldots, \sigma_i.$
By applying the hypothesis iteratively, there exists $\sigma_1,\ldots ,\sigma_i$ with 
\begin{equation}\label{eq: prod i in T}
\left|\prod_{i\in T}A_i^{\sigma_i}\right|\ge \min(m, \prod_{i\in T}|A_i|^{c}) \ge \min(m, \prod_{i\in T}|A_i|^{cc'}).
\end{equation}

Combining the inequalities \eqref{eq:prod i in s} and \eqref{eq: prod i in T} by taking their geometric mean completes the proof. 
\end{proof}

\begin{lem}\label{lem: always growing small sets alternating groups} For every $C>1$ there exists $n_0,c>0$, such that the following holds. Let  $n>n_0$, and let  $G$ be the alternating group $A_n$. Suppose that $A_1,\ldots, A_i \subseteq G$. Then there exists $\sigma_1,\ldots \sigma_i\in G$ with \[|A_1^{\sigma_1} \cdots A_i^{\sigma_i}|\ge \min(|A_1|^{c}\cdots |A_i|^c, e^{Cn}).
\]
\end{lem}
\begin{proof}
Let $\epsilon_0,c'$ be sufficiently small absolute constant, and let  $c = c(\epsilon_0,c', C)>0$ be sufficiently small as a function of $\epsilon_0, C,c_0$. We set $\epsilon = \min(\frac{1}{12C},\epsilon_0)$, and let $\mathcal{E}$ be the collection of sets that are contained in a conjugacy class of support $\le \epsilon \log |A|$. Then 
    by Lemma \ref{lem: growth or small support} there exists $c' = c'(\epsilon)>0$, such that for each set $A$ and each set $B$ of size $<|G|^{0.9}$, either there exists $\sigma$ with $|A^{\sigma} B| > |A||B|^{c'}$ or there exists $a\in A$, such that $a^{-1}A$ contains $|A|^{0.9}$ elements of support $\le \epsilon \log |A|.$ Furthermore, for all $m$ there are only $e^{O(\sqrt{m})}$ conjugacy classes of support $\le m$ in $G$, and therefore the largest intersection of the set $a^{-1}A$ with a conjugacy class of support $\le \epsilon \log |A|$ has size $\ge \frac{|A|^{0.9}}{e^{O(\sqrt{\epsilon \log |A|})}}\ge |A|^{0.8}$ provided that $\epsilon_0$ is sufficiently small. 

    By Lemma \ref{lem: getting rid of the easy either case in either or situations.} it therefore suffices to prove the lemma when each set $A_i$ is contained in a conjugacy class of support $\le \epsilon \log |A_i|$. 
    
    Let $m_i$ be the size of the support of the elements in $A_i$. We may apply Lemma \ref{lem: skew product theorem for elements of small support} iteratively to find $\sigma_1,\ldots, \sigma_i$ with  $|A_1^{\sigma_1}\cdots A_i^{\sigma_i}| \ge |A_1\cdots A_i|^{1/3},$ provided that $m_1+\cdots +m_i \le n/4.$ 
    This complets the proof as when \[m_1 + \cdots +m_i \ge n/4\] we have $|A_1|\cdots |A_i|\ge e^{m_1/\epsilon + \cdots + m_i/\epsilon} \ge e^{\frac{n}{4\epsilon}}$, and therefore  
     \[|A_1^{\sigma_1}\cdots A_i^{\sigma_i}| \ge |A_1|^{1/3}\cdots |A_i|^{1/3} \ge e^{\frac{n}{12 \epsilon}}\ge e^{Cn}.\] 
\end{proof}

\subsection{Growth or concentration in a conjugacy class}
 
We now prove a skew-product lemma that implies growth of the form $|A^{\sigma}B|\ge |A|^{c}|B|$ for an absolute constant $c>0$ or both $A,B$ must be highly concentrated in certain conjugacy classes. 

\begin{lem}\label{lem: growth or concentration in a conjugacy class}
  Let $r$ be an integer and $G$ be a group with $r$ conjugacy classes.  Let $A,B\subseteq G$, suppose that $r^{20} \le |A|$, and suppose further that for every conjugacy class $\mathcal{C},a\in A,b\in B$, \[|\mathcal {C} \cap a^{-1}A||\mathcal {C}^{-1} \cap b^{-1} B| \le |A^{0.9}| |\mathcal{C}|.\] Then there exists $\sigma \in G$ with 
$|A^{\sigma} B| > |A|^{0.05}|B|.$
\end{lem}
\begin{proof}
 Choose $\sigma,\tau$ randomly and independently out of $G$. For each $a\in A,b\in B$ let $X_{a,b}$ be the indicator of the event $a\sigma b = \tau.$ Let $Y$ be the maximum of the $X_{a,b}$ over all $a\in A,b\in B$. Then $Y$ is the indicator of the event $\tau \in A\sigma B.$ 
 
 Taking expectation first over $\sigma$ and then over $\tau$ we obtain $\mathbb{E}[Y] = \mathbb{E}_{\sigma}\left[\frac{|A\sigma B|}{|G|}\right] = \mathbb{E}_{\sigma}\left[\frac{|A^{\sigma} B|}{|G|}\right].$ By the first moment method, in order to prove the lemma, it suffices to lower bound $\mathbb{E}[Y] \ge \frac{|A|^{0.05}|B|}{|G|}.$
 Let $X = \sum X_{a,b}.$ Then we may use the Payley--Zigmund inequality to lower bound $\mathbb{E}[Y] = \Pr[X>0] \ge \frac{\mathbb{E}^2[X]}{\mathbb{E}[X^2]}.$ 
 Let us start by computing first moments. Here we have $\mathbb{E}[X] = \frac{|A||B|}{|G|}$ as for each coice of $a,\sigma ,b$ there is only a single choice of $\tau$ with $\tau =a\sigma b.$  

 The second moment computation relies on the computation of expectations of the form $\mathbb{E}[X_{a,b}X_{a'b'}]$. 
 Now the conditional expectation $\mathbb{E}[X_{a,b}X_{a'b'} | \sigma]$ is 0 whenever $a\sigma b\ne a'\sigma b'$ and $\frac{1}{|G|}$ otherwise. Moreover, $a\sigma b =a'\sigma b'$ if and only if $\sigma^{-1} a^{-1} a' \sigma  = b b'^{-1}.$ This can happen only if $a^{-1}a'$ and $bb'^{-1}$ are in the same conjugacy class $\mathcal{C}$, and the equality then happens with probability $\frac{1}{|\mathcal{C}|}.$ 
 Fixing $a\in A,b\in B$,
 this shows that 
 \[
 \sum_{a',b'} \mathbb{E}[X_{a,b}X_{a',b'}] = \sum_{\mathcal{C}} \sum_{a' \in a \mathcal{C}\cap A, b' \in b\mathcal{C}^{-1}\cap B} \frac{1}{|G||\mathcal{C}|} = \sum_{\mathcal{C}} \frac{|a\mathcal{C}\cap A| |b\mathcal{C}^{-1} \cap B|}{|\mathcal{C}||\mathcal{G}|}.\]
Plugging in the bound $|a\mathcal{C}\cap A| |b\mathcal{C}^{-1} \cap B| \le |A|^{0.9}|\mathcal{C}|$ we obtain  
\[
\sum_{a',b'} \mathbb{E}[X_{a,b}X_{a',b'}]\le \frac{r|A^{0.9}|}{|G|}\le \frac{|A|^{0.95}}{|G|}.
\]
 Summing over all $a\in A, b\in B$ we obtain that 
 $\mathbb{E}[X^2]\le |A|^{0.95}\mathbb{E}[X].$ The Payley--Zigmund inequality now implies that $\Pr[X>0] \ge \frac{\mathbb{E}[X]}{|A|^{0.95}} \ge \frac{|A|^{0.05}|B|}{|G|}.$ This completes the proof of the lemma. 
\end{proof}

\subsection{Growth or concentration in a coset}
In this section our goal is to show that for every set $A$ of size $\ge e^{Cn}$, we either have $|A^{\sigma}B|>|A|^{c}|B|$ for an absolute constant $c>0$ or $A$ has a large subset contained in a coset of the pointwise stabilizer of a large set. 

The idea is simple. We already know that a `large chunk' of a translate of $A$ is contained in a small conjugacy class by Lemma \ref{lem: growth or concentration in a conjugacy class}. We can then make use of the fact that a conjugacy class of support $s$ is contained in the union of the $\binom{n}{s}$ pointwise stabilizers of sets of size $n-s$. 
\subsubsection*{$t$-umvirates}
For a set $I\subseteq [n]$ of size $t$ and an injection $\psi \colon I\to [n]$ the set of permutations $\tau$ such that $\tau(i) =\sigma(i)$ for all $i\in I$ is called an $t$-umvirate, and we denote it by  $U_{I,\sigma}.$ The $t$-umvirate $U_{I,\sigma}$ is a coset of the pointwise stabilizer of the set $I$. 

We write $U_{I}$ for the $t$-umvirate of permutations fixing the set $I$. We also identify $U_I$ with the set of even permutation of the set $[n]\setminus I.$ We therefore also denote $S_{[n]\setminus I}$ for $U_{I}.$ For disjoint sets $I_1,\ldots, I_k$ we write 
$S_{i_1} \times \cdots \times S_{I_k}$ for the product of the subgroup $S_{I_j}$ inside the symmetric group $S_n$   

\begin{lem}\label{lem: growth or coset}
There exists  $C, n_0 >0$, such that the following holds. 
Let $n>n_0,$ and let $A,B\subseteq A_n$ be with $C^n \le |A|.$ Suppose that $|A^{\sigma} B|<|A|^{0.05}|B|$ for all $\sigma \in A_n$. Then there exists $t$ and $t$-umvirates $U,U'$, such that $|A \cap U|> |A|^{0.8}$ and $|U' \cap B| \ge |U'|^{0.2}$.  
\end{lem}
\begin{proof}
By Lemma \ref{lem: growth or concentration in a conjugacy class} there exists a conjugacy class $\mathcal{C}$ with  \[
|a^{-1} A \cap \mathcal{C}| |b^{-1}B\cap \mathcal{C}^{-1}|\ge |\mathcal{C}||A|^{0.9}.\] 

In particualar,
\begin{equation}\label{eq:A is large inside C}
|a^{-1} A \cap \mathcal{C}| \ge |A|^{0.9} \frac{\mathcal{|C|}}{|b^{-1}B\cap \mathcal{C}^{-1}|} \ge |A|^{0.9} 
\end{equation}
and \[|b^{-1}B\cap \mathcal{C}^{-1}|\ge |\mathcal{C}|\frac{|A|^{0.9}}{|a^{-1} A \cap \mathcal{C}|} \ge  \frac{|\mathcal{C}|}{|A|^{0.1}}.\]
As \eqref{eq:A is large inside C} implies that $|\mathcal{C}|>|A|^{0.9}$ we obtain that 
\[
|b^{-1}B\cap \mathcal{C}^{-1}|\ge |\mathcal{C}|^{1 -\frac{0.1}{0.9}}=|\mathcal{C}|^{8/9}.
\]
Moreover, we have $\mathcal{|C|}>|A|^{0.9}>C^{0.9n}.$

Let $t$ be the number of fixed points in a permutation in $\mathcal{C}.$ Then $a^{-1} A \cap \mathcal{C}$ is contained in the union of the sets $a^{-1}A \cap U_{I}$ over all sets $I$ of size $t$ and similarly for $b^{-1}B.$ 
Therefore, there exists a set $I$ of size $t$, such that  \[|a^{-1} A \cap U_{I}| \ge \frac{|A|^{0.9}}{\binom{n}{t}}\ge |A|^{0.8},\] provided that $C$ is sufficiently large.

Similarly, there exists a set $J$ of size $t$, such that $|b^{-1}B\cap U_{J}|\ge \frac{|\mathcal{C}|^{8/9}}{\binom{n}{t}}\ge |\mathcal{C}|^{0.8},$ provided that $C$ is sufficiently large.

We now make use of the fact that if $\mathcal{C}$ is a conjugacy class of permutations with $t$-fixed points, then the orbit stabilizer theorem implies that for sufficiently large $t$, we have $t^{t/3}<\binom{n}{t} t^{t/3} < |\mathcal{C}|$. Therefore,

\[
|b^{-1}B\cap U_{J}| \ge |U_{J}|^{\frac{0.8}{3}}\ge |U_{J}|^{0.2}.
\]

Setting $U = aU_I$ and $U' = bU_J$ completes the proof.
\end{proof}

\subsection{Growth from $e^{Cn}$ to $e^{n\log^{3/4} n}$ }

We now prove a growth result that allows us to obtain growth in several steps from from $2^{Cn}$ to $2^{C'n\sqrt{\log n}}$ and then again from $2^{C'n\sqrt{\log n}}$ to $2^{C'n\log^{3/4} n}$. The idea is to use Lemma \ref{lem: growth or coset} to reduce to the case where the sets $A_i$ are contained in $t$-umvirates and then to give a growth result for sets contained in small $t$-umvirates.

\begin{lem}\label{lem: medium sets always growing alternating groups}
There exists $C,c,n_0>0$, such that the following holds. Let $n<n_0,r>Cn$ and let $G=A_n$ be the alternating group. Suppose that $A_1,\ldots,A_i\subseteq G$ are subsets having size $\ge e^r$.  Then there exist $\sigma_1,\ldots, \sigma_i$  with  $|A_1^{\sigma_1}\cdots A_i^{\sigma_i}|\ge \min(|A_1|^{c} \cdots |A_i|^{c}, e^{c \sqrt {n\log n r}})$. 
\end{lem}
\begin{proof}
Let $c>0$ be a sufficiently small constant. 
By Lemma \ref{lem: growth or coset} for each set $A$ either for all $B$ of size $\le n^{c \sqrt {nr}})$ there exists $\sigma$ with 
\[|A^{\sigma}B|\ge |A|^{0.05}|B|\]
or there exists a $t$ and $t$-umvirates $U,U'$, such that $|A\cap U|>|A|^{0.8}$ and $|U'\cap B|\ge |U|^{0.2}.$ In particular, $|U|<|B|^5 <e^{5c\sqrt{n\log n r}}$. 
Let $\mathcal{E}$ be the collection of sets contained in $t$-umvirates $U$ with $|U|<e^{5c\sqrt{n\log n r}}.$
Then by Lemma \ref{lem: getting rid of the easy either case in either or situations.} we may assume that $A_j\in \mathcal{E}$ for all $j$.
Let us denote the corresponding $t_j$-umvirate by $S_{I_j}$. So we have $A_j\subseteq S_{I_j}$ and  $|I_j| = n-t_j$.

 Then \[ |I_j|!  = |U_{I_j}| <e^{5c \sqrt {n\log n r}}. \] 
 This yields the existence of absolute constants $C',C''>0$, such that 
 \[
 |I_j| \le C' \frac{ 5c \sqrt {n\log n r}}{\log( 5c \sqrt {n\log n r})} \le C'' c \sqrt{n r/\log n}\le \sqrt{nr/\log n}/2,
 \]
 provided that $c$ is sufficiently small. 

Let $t$ be the minimal value of the $t_j$, and set $m = \lfloor \frac{n}{n-t}\rfloor$. Furthermore, let $\sigma_1 ,\ldots ,\sigma_m$ be even permutations, such that the sets $\sigma_j(I_j)$ are pairwise disjoint. Then the conjugate $A_j^{\sigma_j^{-1}}$ of $A_j$
is contained in the $|I_j|$-umvirate $S_{I_j}^{\sigma_j^{-1}} = S_{\sigma_j(I_j)}$. 

Since the product map 
$ (a_1,\ldots a_m) \mapsto a_1 \cdots a_m$ embeds the direct product $S_{\sigma_1(I_1)} \times \cdots \times S_{\sigma_m(I_m)}$ inside $S_n$, and since each set $A_j^{\sigma_j}$ is contained in $S_{\sigma_j(I_j)}$
it follows that  $|A_1^{\sigma_1}\cdots A_i^{\sigma_i}| = |A_1|\cdots|A_i|$. 

Since \[|A_1|\cdots |A_m| \ge e^{r \lfloor n/(n-t)\rfloor }\ge e^{\frac{rn}{\sqrt{nr/\log n}}} = e^{\sqrt{r n\log{n}}},\] this completes the proof of the lemma, provided that $c<1$. 
\end{proof} 

\subsection{Growth or covering by $t$-umvirates}
Earlier we mainly argued that if unless we have growth of the form $|A^{\sigma} B|\ge |A|^{c}|B|$ for some absolute constant $c>0$, the set $A$ must have a unique structure. We now move on to analyze the structure of the set $B$. 

\begin{lem}\label{lem: growth or covering by cosets}
There exists  $C,c,n_0 >0$, such that the following holds. 
 Let $n>n_0,\epsilon>0$ and let $A,B\subseteq A_n$ be with $C^n \le |A|.$ Suppose that $|A\sigma B|<\frac{1}{2}|A|^{0.05}|B|$ for all $\sigma \in A_n$. Then there exist $t$, a set $I$ of size $t$,  a set $S$ of permutations of the set $I$ an element $b\in B,$ and $r$ in the interval $[0.2, 1]$ such that
 \begin{enumerate}
 \item The intersection size $|U_{I, \sigma} \cap b^{-1} B|$ is in the interval $ \left[|U_{I, \sigma}|^{r} ,|U_{I,\sigma}|^{r+\epsilon}\right]$ for every $\sigma \in S$.
 \item  $|U_{I, \sigma }| > |A|^{0.8}$
 \item $\left|b^{-1}B \cap \bigcup_{\sigma \in S}U_{I, \sigma }\right| \ge \frac{\epsilon |B|}{2n \binom{n}{t}^2}$ 
 \end{enumerate}
\end{lem}
\begin{proof}
    Suppose that $\frac{1}{2}|A|^{0.05}|B|$
    Define iteratively a sequence of sets \[B = B_0\supseteq  B_1\supseteq \cdots \supseteq B_m\] as follows. In the $i$th stage, where $B_i$ is defined you stop the process and set $m=i$ if $|B_i|<|B|/2$.  Otherwise for all $\sigma$ we have 
    $|A^{\sigma} B_{i}|<|A^{\sigma} B|<|A|^{0.05}|B_i|.$ 
    By Lemma \ref{lem: growth or coset} there exists $t_i$ and a $t_i$-umvirate $U_{I_i,\sigma_i}$, such that $|U_{I_i,\sigma_i}|>|A|^{0.8}$ and 
    $|B_{i}\cap U_{I_i,\sigma_i}| > |U_{I_i,\sigma_i}|^{0.2}$.
    We then set $B_{i+1} = B_{i}\setminus U_{I_i,\sigma_i}$. 

    By construction we have $|B\cap \bigcup_{i=1}^{m} U_{I_i,\sigma_i }| \ge |B|/2.$
    As there are at most $n$ potential choices for the $t_i$ this shows that there exists $t$, such that when setting $S'\subseteq [m]$ to be the set of values $i$ for which $U_{I_i,\sigma_i}$ is a $t$-umvirate, we have $|B\cap \bigcup_{i\in S'} U_{I_i,\sigma_i }| \ge \frac{|B|}{2n}.$
    Partitioning the interval $[0.2,1]$ into at most $1/\epsilon$ intervals of size $\le \epsilon,$ we obtain that there exists $r$, such that when setting $S''\subseteq S'$ to be the set of all $i$ with  $|B\cap U_{I_i,\sigma_i}| \in \left[|U_{I,\sigma}|^r, |U_{I,\sigma}|^{r+\epsilon}\right]$ we obtain that 
    \[
    |B\cap \bigcup_{i\in S''} U_{I_i,\sigma_i }| \ge \frac{\epsilon|B|}{2n}.
    \]
    Since there are at most $\binom{n}{t}$ coices of $I_i$ and $\binom{[n]}{t}$ choices for the set $\sigma_i(I_i)$ we may find sets $I,J$ of size $t$, such that when setting $S'''$ to be the set of $t$-umvirates  $U_{I_i,\sigma_i}$ for which $|I_i| = I$ and $\sigma_i(I_i) =J$ we have 
    \[
        |B\cap \bigcup_{U \in S'''} U| \ge \frac{\epsilon|B|}{2n\binom{n}{t}^2}.
    \]
    Let $b \in B\cap \bigcup_{U \in S} U.$ Then $| b^{-1}B \cap \bigcup_{U \in S'''} b^{-1}U| \ge \frac{\epsilon|B|}{2n\binom{n}{t}^2}.$ 
    
    So setting $S$ to be the set of permutations of $I$ of the form $b^{-1}\sigma$ with $U_{I,\sigma}\in S'''$ completes the proof.
\end{proof}

It will be convenient to work with the following simplification of the lemma that corresponds to a certain choice of the parameters. 

\begin{lem}\label{lem: simplifiel growth or covering}
    There exists  $C,c,n_0 >0$, such that the following holds. 
 Let $n>n_0,\epsilon>0$ and let $A,B\subseteq A_n$ be with $2^{Cn\sqrt{\log n}} \le |A|.$ 
 Suppose that 
 \[|A\sigma B|<\frac{1}{2}|A|^{0.05}|B|\]
 for all $\sigma \in A_n$. Then there exist a set $I$ of size $\le n - \frac{n}{\sqrt{\log n}},$ $b\in B, r\in [0.2,1)$ and a set $B' \subseteq b^{-1}B \cap S_{I}\times S_{[n]\setminus I},$ such that the following holds. Let $\pi$ be the projection map from $S_{[n]\setminus I} \times S_I$ to $S_{I}.$ Then \[|\pi(B')|\ge \frac{|B|}{5^n (n-|I|)!^r}.\] Moreover, for every $\sigma \in \pi(B')$, we have \[|\pi^{-1}(\sigma)\cap B'| \ge (n-|I|)!^r.\]  
\end{lem}
\begin{proof}
    We may apply Lemma \ref{lem: growth or covering by cosets} with $\epsilon = 1/n$ to find $t,$ a set $I$ of size $t$, and a set $S$ of permutations in $S_{I}$ with 
   $|b^{-1}B \cap \bigcup_{\sigma\in S} U_{I,\sigma}|\ge \frac{|B|}{2n^2\binom{n}{t}^2},$ such that 
   $|U_{I}|>|A|^{0.8}$ and also 
   $|B\cap U_{I,\sigma}| \in \left[ |U_{I,\sigma}|^{r}, |U_{I,\sigma}|^{r+1/n} \right] $ for all $\sigma \in S.$ 
   Let $B' = b^{-1}B \cap \bigcup_{\sigma\in S} U_{I,\sigma}.$ Then $B'\subseteq S_I \times S_{[n]\setminus I}$ and  every $\sigma \in \pi(B')$, is in $S$. This shows that 
   \[|\pi^{-1}(\sigma) \cap B'|  = |U_{I,\sigma}\cap B'|\in [(n-|I|)!^r, (n-|I|)!^{r+1/n}.\]
   Now $|B'|\ge \frac{|B|}{2n^2\binom{n}{t}^2}$ 
   in conjunction with the upper bound  $|\pi^{-1}(\sigma) \cap B'|\le (n-|I|)!^{r+1/n}$ for each $\sigma\in \pi(B')$, imply that  
   \[|\pi(B')|\ge \frac{|B|}{2n^2\binom{n}{t}^2 |(n-|I|)!|^{r+1/n}}\ge \frac{|B|}{|(n-|I|)!|^{r} 5^n},\] provided that $n$ is sufficiently large. 
   Finally, as 
   \[
   2^{0.8C n \sqrt{\log n}}\le |A|^{0.8} \le |U_{I}| = (n-|I|)!.
   \]
   We obtain that 
   \[
   n-|I| \ge n/\sqrt{\log n},
   \]
   provided that $C$ is sufficiently large.
\end{proof}

\subsection{Growth up to $|G|^{c}$}

\begin{lem}\label{lem: skew product for large sets}
There exists $C,c,n_0>0$, such that if $n>n_0,$  and $A\subseteq A_n$ has size $\ge 2^{C\sqrt{n}\log n},$ then $|A^{\sigma}B|>\min(|A|^{c}|B|, |G|^{c})$ for all subsets $B$ of $A_n$.
\end{lem}
\begin{proof}
By Lemma \ref{lem: growth or coset} there exist $t$ and a $t$-umvirate $U$ with $|A\cap U|>|A|^{0.8}$ and $(n-t)! = |U|<|B|^{5}< n^{5cn}.$ Provided that $c$ is sufficiently small, this implies that $t\ge n/2.$ By decreasing $c$ is necessarily it is sufficient to prove the lemma for $A\cap U$ replacing $A$.
    So without loss of generality let us assume that $A$ is contained in an $\lfloor n/2 \rfloor$-umvirate $U$. 

    By Lemma \ref{lem: simplifiel growth or covering} there exists a set $I$ of size $\le n(1-\frac{1}{\sqrt{\log n}})$, a permutation $b\in B,$ $r\in [0.2,1],$ and a subset $B'\subseteq b^{-1}B \cap S_I \times S_{n\setminus I},$ 
    such that $B'\cap \pi^{-1}(\{\sigma\}) \ge (n-|I|)!^{r}$ for all $\sigma \in \pi(B')$ and the size of $B_1: = \pi(B')$ is at least $\frac{|B|}{5^n (n-|I|)!^r}$. 

    Now note that if $|I|\ge n/2,$ then $|B|\ge (n-|I|)!^{r}\ge n^{cn}$, provided that $c$ is sufficently small, which is a contradiction.  Therefore $|I|<n/2.$
    As $A$ is contained in an $\lfloor n/2\rfloor$-umvirate, we may find a translate $A_1:= \sigma_1 A\sigma_2$ of it contained in $S_{[n]\setminus I}$. Set $I_1 = I, J_1=[n]\setminus I_1.$
    
    While we can, we iterate with $A_1, B_1,I_1$ taking the roles of $A_0 = A, B_0 = B, I_0 = [n]$ and repeat to find tuples $(A_i,  B_i, I_i, J_i,r_i)$ stopping when either $|I_i|<n/2$ or $|A_i ^{\sigma} B_i|\ge \frac{1}{2}|A_i|^{0.05}|B|$. 
    
    The tuples $(A_i,  B_i, I_i, J_i,r_i)$ have the following properties. Firstly, the set $A_i$ is a translate of $A$ contained in $S_{I_i}$. Moreover, we have $J_i = J_{i-1} \setminus I_i,$ and 
    \[|I_i|\le |I_{i-1}|\left(1 - \frac{1}{\sqrt{\log |I_{i-1}|}}\right)\le |I_{i-1}|\left(1  - \frac{1}{\sqrt{\log (n/2)}}\right),\] 
    and furthermore there exists $b\in B_{i-1},$ and $B' \subseteq b^{-1} B_{i-1} \cap S_{I_i}\times S_{J_i}$, such that when setting $\pi_i$ to be the projection map $\pi_i \colon S_{I_i}\times S_{J_i} \to S_{I_i}$ we have $\pi_{i}^{-1}(\sigma)\ge |J_i|!^{r_i}$ for every $\sigma\in B_i:= \pi(B')$ and  
    \[
    |B_i| \ge \frac{|B_{i-1}|}{5^{|J_i|}|J_i|!^{r_i} }
    \] 

    As this process is stopped whenever $|I_i| \le n/2$, and as $I_i \le I_{i-1}(1-\frac{1}{\sqrt{\log n}})$, this process is iterated at most $C' \sqrt{\log n}$ times for some absolute constant $C'.$ This guarantees that $|B_i|\ge \frac{|B|}{|J_1|!\cdots |J_r|! \cdot 5^{C'n\sqrt{\log n}}}$ throuout the process. 
    Moreover, there exists a set  $B'\subseteq b^{-1}B\cap S_{I_i}\times S_{J_i} \times S_{J_{i-1}}\cdots S_{J_1}$, such that $B_i$ is the projection of $B'$ onto $S_{I_i}$ and for the projection map \[\pi\colon S_{I_i}\times S_{J_i} \times S_{J_{i-1}}\cdots S_{J_1} \to S_{I_i}\] we have  \[|\pi^{-1}(\{\sigma\}) \cap B'|\ge |J_i|!^{r_i} \times |J_{i-1}|!^{r_i-1} \cdots |J_1|!^{r_1}.\] 

    Now the process terminates either for $|I_i|\le n/2$, or when 
    \[|A_j ^{\sigma} B_j|>1/2|A_j|^{0.05}|B_j|\] for some $\sigma \in S_{I_j}.$ Since, $|J_j| \ge \frac{n}{c'\sqrt{\log n}}$ for an absolute constant $c'>0$ we have  
    \[
    |J_{j}|!\cdots |J_1|!\ge n^{c''(|J_1| +\cdots + |J_{j}|)} \ge n^{c''n/2}.
    \]
    Thus,
    \[
    |B| \ge \frac{n^{c''n/2}}{5^{C'n\sqrt{\log n}}}> n^{cn},
    \] provided that $c$ is sufficiently small with respect to $C',c''$, which is a contradiction. Therefore, this process terminates when for some $\sigma \in S_{I_j}$, 
    \[
    |A_i^{\sigma} B_i| > \frac{1}{2}|A_i|^{0.05} |B_i| = |A|^{0.05} |B_i|,
    \]     
   provided that $C$ is sufficiently large. 
   Let $B' \subseteq b^{-1}B \cap S_{I_i}\times S_{J_i}\times \cdots \times S_{J_1}$ be such that for every $\sigma \in S_{I_i}$, we have
   \[|\pi^{-1}(\{\sigma\}) \cap B'| \ge 
   |J_i|!^{r_i} \times |J_{i-1}|!^{r_i-1} \cdots |J_1|!^{r_1},\]
   and such that $\pi(B') = B_i$.  We now obtain that \[ |B^{\sigma} A| = |b^{-1}B^{\sigma} A| \ge |B'^{\sigma} A| \ge |\pi(B'^{\sigma} A)| |J_i|!^{r_i} \times |J_{i-1}|!^{r_i-1} \cdots |J_1|!^{r_1}. \]
   Now provided that $C$ is sufficiently large the left hand side is at least \(
    |B||A|^{0.01},
    \)
    which completes the proof.
\end{proof}

We obtain the following lemma as an immediate corollary. 
\begin{lem}\label{lem: Growth for large sets}
There exists $C,c,n_0>0$, such that if $n>n_0$ and $G=A_n$. Then for every subsets $A_1\subseteq A_i\subseteq G$ of size $\ge 2^{C\sqrt{n}\log n},$ there exist $\sigma_1,\ldots, \sigma_i \in G$ with  $|A_1^{\sigma_1} \cdots A_i^{\sigma_i}| \ge \min(|A_1 |^c \cdots |A_i|^c, n^{cn}).$
\end{lem}
\begin{proof}
    The statement follows by iterating Lemma \ref{lem: skew product for large sets}.
\end{proof}

\subsection{Completing the proofs}

We now complete the proof of Theorem \ref{thm:alternating growth}, and the Liebeck--Nikolov--Shalev conjecture.  
We will make use of the following concatenation claim. It states that a lemma yielding growth  of $|A_1^{\sigma_1}\cdots A_{i}^{\sigma_i}|$ up to size $\ell$, and a second lemma implying growth up to size $m$, but only when the sets $A_i$ are of size $\ge \ell$, can be concatenated to yield growth up to size $m$ for all sets $A_i$.

\begin{claim}\label{claim:concatenation}
Let $G$ be a group and let $\ell <m$ be integers and $c\in (0,1)$. Suppose that for all $k$ and all sets $A_1,\ldots,A_k$ there exist $\sigma_1,\ldots, \sigma_k$ with  
\[
|A_1^{\sigma_1} \cdots A_k^{\sigma_k}|\ge \min(|A_1|^c \cdots |A_k|^{c}, \ell).
\]
Suppose further that for all $k$ and all sets $A_1,\ldots , A_k$ of size $\ge \ell$ there exist $\sigma_1,\ldots ,\sigma_k$ with 
$|A_1^{\sigma_1}\cdots A_k^{\sigma_k}|\ge \min(|A_1|^c \cdots |A_k|^{c}, m).$ Then for all $k$ and all sets $A_1,\ldots, A_k$ there exist $\sigma_1,\ldots, \sigma_k\in G$ with 
\[
|A_1^{\sigma_1}\cdots A_k^{\sigma_k}|\ge \min(|A_1|^{c^2/4} \cdots |A_k|^{c^2/4}, m).
\]
\end{claim}
\begin{proof}

We may find consecutive intervals of integers  $I_1, \ldots , I_r, I'$ covering $\{1,\ldots, k\}$, such that each interval $I_j$ is minimal with respect to $\prod_{i\in I_j}|A_i|^{c} \ge \ell,$ and where $\prod_{i\in I'}|A_i|^c < \ell.$ 
If $r= 0,$ and $I' = \{1,\ldots , k\},$ then the statement clearly holds. So suppose that $r>1$, and let $\ell_j$ by the last number in the interval $I_j$. 
Then by hypothesis for all $j$ there exists a set $B_j$ of the form \(
B_j:= \prod_{i\in I_j}A_i^{\sigma_i}
\)
 with 
 \[
 |B_j|\ge \ell \ge \prod_{i\in I_j\setminus{\ell_j}}|A_i|^{c} 
 \]
 and also trivially
 \(
    |B_j|\ge |A_{\ell_j}|\ge |A_{\ell_j}|^{c}.
 \)
 Taking geometric mean we have  
\[
|B_j|\ge \max(\ell, \prod_{i\in I_j} |A_i|^{c/2}).
\]

 We may now apply the hypothesis with the sets $B_i$ to obtain that there exits $\tau_1,\ldots, \tau _r$ with \[|B_1^{\tau_1}\cdots B_{r}^{\tau_r}|\ge \min (|B_1|^{c}\cdots |B_r|^{c}, m) \ge \min(m , \prod_{j=1}^r\prod_{i\in I_j} |A_i|^{c^2/2}.) 
 \] 
 This completes the proof as 
 \[
 \prod_{j=1}^r\prod_{i\in I_j} |A_i| \ge \sqrt{\prod_{i=1}^{k} |A_i|}.
 \]
 
 Indeed, $\prod_{i\in I'}|A_i|^c \le \ell \le \prod_{i\in I_1} |A_i|^c.$  
\end{proof}

\begin{proof}[Proof of Theorem \ref{thm:alternating growth}]
Let $c_1$ be a sufficiently small constant and let $C$ be sufficiently large with respect to $c_1$ and finally, let $c$ be sufficiently large with respect to $C$.  
    Let $\ell_1 = C^n, \ell_2 = 2^{c_1 C n \sqrt{\log n}}, \ell_3 = 2^{c_1 C^2 n\log^{3/4} n}, \ell_4 = n^{cn}$. The theorem follows from Claim \ref{claim:concatenation} applied multiple times, where Lemma \ref{lem: always growing small sets alternating groups} implies the corresponding growth result up to $\ell_1$, and Lemma \ref{lem: medium sets always growing alternating groups} then yields the growth result from  $\ell_1$ to $\ell_2$. Combining these Lemma with Claim  Claim \ref{claim:concatenation} yields a growth result up to $\ell_2$. Moreover, Lemma \ref{lem: medium sets always growing alternating groups} yields a growth result from $\ell_2$ to $\ell_3$ and Lemma \ref{lem: Growth for large sets} yields growth from $\ell_3$ up to $\ell_4.$ So combining these with Claim \ref{claim:concatenation} applies twice completes the proof of the lemma. 
\end{proof}

\begin{proof}[Proof of Theorem \ref{thm:main intro}]
    For finite simple groups of Lie type this follows from Theorem \ref{thm: Lie type LNS}. To complete the proof for alternating groups note that by Theorem \ref{thm:alternating growth} there exists an absolute constant $c'>0,$ such that there are $i$ conjugates of $A$ whose product has size at least $\min(|A|^{c'i}, n^{c'n}).$ By Theorem \ref{thm: GLPS main theorem} there exist a constant $C= C(c')>0$, such that for all subets $B$ of size $\ge n^{c'n}$ there are $C$ conjugates of $B$ whose product is $A_n.$ This in particular shows that there are $C\left\lceil \frac{\log |G|}{c'\log |A|}\right\rceil$ conjugates of $A$ whose product is the whole group, thereby completing the proof of the original form of the Liebeck--Nikolov--Shalev conjecture. 
    
    This also completes the proof of our stronger version with $c = \frac{c'}{2C}.$ Indeed, when $|A|^{ci} \le n^{cn}$ there are $i$ conjugates of $A$ whose product has size at least $\min(|A|^{c'i}), n^{cn}) \ge \min(|A|^{ci}, n^{cn}) = |A|^{ci}.$ On the other hand, when  $|A|^{ci} \ge n^{cn}$ we have $i\ge C\left\lceil \frac{\log |G|}{c'\log |A|}\right\rceil$ and therefore there are $i$ conjugates of $A$ whose product is the whole group.
\end{proof}

\bibliographystyle{alpha}
\bibliography{refs}
\end{document}